\documentclass[11pt, DIV=12]{scrarticle}
\usepackage{iftex}
\ifPDFTeX
	\usepackage[utf8]{inputenc}
\else
	\usepackage[utf8]{luainputenc}
\fi
\usepackage[T1]{fontenc}
\usepackage[UKenglish]{babel}
\usepackage{amsmath, amsfonts, amsthm, amssymb, mathtools, stmaryrd}

\usepackage{tikz}
\usepackage{tikz-cd}
\usepackage{tkz-graph}
\usetikzlibrary{matrix,arrows,backgrounds,positioning,fit,cd,babel}

\usepackage[pdfusetitle,hidelinks]{hyperref}
\usepackage[capitalise,noabbrev]{cleveref}

\usepackage{csquotes}
\usepackage[osf,sc]{mathpazo}
\addtokomafont{disposition}{\normalfont\bfseries} \usepackage{relsize}
\usepackage{microtype}

\usepackage{thm-restate}
\usepackage{subcaption}
\usepackage{enumitem}
\usepackage{comment}

\definecolor[named]{lipicsGray}{rgb}{0.31,0.31,0.33}
\definecolor[named]{lipicsBulletGray}{rgb}{0.60,0.60,0.61}
\definecolor[named]{lipicsLineGray}{rgb}{0.51,0.50,0.52}
\definecolor[named]{lipicsLightGray}{rgb}{0.85,0.85,0.86}
\definecolor[named]{lipicsYellow}{rgb}{0.99,0.78,0.07}

\recalctypearea

\usepackage{orcidlink}
\usepackage{booktabs}

\theoremstyle{definition}
\newtheorem{definition}{Definition}[section]

\newtheorem{example}[definition]{Example}

\theoremstyle{plain}
\newtheorem{lemma}[definition]{Lemma}
\newtheorem{corollary}[definition]{Corollary}
\newtheorem{theorem}[definition]{Theorem}
\newtheorem{conjecture}[definition]{Conjecture}

\newtheorem{observation}[definition]{Observation}
\newtheorem{fact}[definition]{Fact}
\Crefname{fact}{Fact}{Facts}
\Crefname{conjecture}{Conjecture}{Conjectures}

\theoremstyle{remark}
\newtheorem{claim}[definition]{Claim}
\Crefname{claim}{Claim}{Claims}
\setlist[enumerate, 1]{font=\upshape, noitemsep, nolistsep}
\setlist[enumerate, 2]{font=\upshape, noitemsep, nolistsep}
\setlist[itemize, 1]{noitemsep, nolistsep,font=\upshape}
\setlist[itemize, 2]{noitemsep, nolistsep,font=\upshape}

\newcommand\symdiff{\mathbin{\triangle}}

\DeclareMathOperator{\tw}{tw}
\DeclareMathOperator{\pw}{pw}

\DeclareMathOperator{\soe}{soe}

\DeclareMathOperator{\cl}{cl}
\DeclareMathOperator{\dist}{dist}
\DeclareMathOperator{\id}{id}
\DeclareMathOperator{\ed}{ed}
\DeclareMathOperator{\dd}{dd}

\usepackage{todonotes}

\usetikzlibrary{matrix,arrows,backgrounds,positioning,fit,cd,babel,calc}

\newenvironment{claimproof}[1][Proof of Claim]{\begin{proof}[#1] }{ \end{proof}}

\title{Distinguishing Graphs by Counting\\Homomorphisms from Sparse Graphs\footnote{An extended abstract of this article was presented at the \emph{41st Annual Symposium on Logic in Computer Science} \cite{lics_version}. }}
\author{Daniel Neuen \orcidlink{0000-0002-4940-0318} \\ \small Technische Universität Dresden \and Tim Seppelt \orcidlink{0000-0002-6447-0568} \\ \small IT-Universitetet i K{\o}benhavn}

\renewcommand{\phi}{\varphi}
\renewcommand{\epsilon}{\varepsilon}

\usepackage{abstract}

\begin{document}
	\maketitle

	\begin{abstract}
		Lovász (1967) showed that two graphs $G$ and $H$ are isomorphic if, and only if, they are \emph{homomorphism indistinguishable} over all graphs, i.e., $G$ and $H$ admit the same number of homomorphisms from every graph $F$.
Subsequently, a substantial line of work studied homomorphism indistinguishability over restricted graph classes.
For example, homomorphism indistinguishability over minor-closed graph classes $\mathcal{F}$ such as the class of planar graphs, the class of graphs of treewidth~$\leq k$, pathwidth~$\leq k$, or treedepth~$\leq k$, was shown to be equivalent to quantum isomorphism and equivalences with respect to counting logic fragments, respectively.

Via such characterisations, the distinguishing power of e.g.\ logical or quantum graph isomorphism relaxations can be studied with graph-theoretic means.
In this vein, 
Roberson (JCTB 2026) conjectured that homomorphism indistinguishability over every graph class excluding some minor is not the same as isomorphism.
We prove this conjecture for all vortex-free graph classes. 
In particular, homomorphism indistinguishability over graphs of bounded Euler genus is not the same as isomorphism.
As a negative result, we show that Roberson's conjecture fails when generalised to graph classes excluding a topological minor.

Furthermore, we show homomorphism distinguishing closedness for several graph classes including all topological-minor-closed and union-closed classes of forests, and prove that homomorphism indistinguishability over graphs of genus $\leq g$ (and other parameters) forms a strict hierarchy.

 	\end{abstract}

	\section{Introduction}
	
Lov\'asz~\cite{lovasz_operations_1967} showed that two graphs $G$ and $H$ are isomorphic if, and only if, they admit the same number of homomorphisms from every graph $F$.
Decades later, Dvo\v{r}\'ak~\cite{dvorak_recognizing_2010} proved that two graphs  $G$ and $H$  satisfy the same sentences in the $k$-variable fragment $\mathsf{C}^k$ of first-order logic with counting quantifiers if, and only if, they admit the same number of homomorphisms from every graph $F$ of treewidth $< k$.
Finally, a celebrated theorem of Man\v{c}inska and Roberson~\cite{mancinska_quantum_2020} asserts that two graphs  $G$ and $H$  are quantum isomorphic if, and only if, they admit the same number of homomorphisms from every planar graph~$F$.

These results motivate the following notion: Two graphs $G$ and $H$ are \emph{homomorphism indistinguishable} over a graph class~$\mathcal{F}$, in symbols $G \equiv_{\mathcal{F}} H$, if they admit the same number of homomorphisms $\hom(F, G) = \hom(F, H)$ from every graph $F \in \mathcal{F}$.
In addition to the initially mentioned results, 
homomorphism indistinguishability over graph classes such as 
the class of graphs of pathwidth $\leq k$, treedepth $\leq k$, the class of outerplanar graphs, the class of graphs of maximum degree $\leq d$ etc.\ has been characterised in terms of graph isomorphism relaxations from areas ranging from finite model theory 
\cite{dvorak_recognizing_2010,dell_lovasz_2018,grohe_counting_2020,fluck_going_2024,schindling_homomorphism_2025,montacute_pebble-relation_2024,dawar_lovasz-type_2021,atserias_expressive_2021,schindling_homomorphism_2025}, 
optimisation \cite{dell_lovasz_2018,grohe_homomorphism_2025,atserias_sherali-adams_2012,grohe_pebble_2015,malkin_sheraliadams_2014,roberson_lasserre_2024,kar_npa_2025},
category theory \cite{montacute_pebble-relation_2024,dawar_lovasz-type_2021,abramsky_discrete_2022,schindling_homomorphism_2025},
algebraic graph theory \cite{dell_lovasz_2018,grohe_homomorphism_2025,rattan_weisfeiler-leman_2023},
and invariant theory \cite{cai_planar_2024,young_converse_2025,cai_vanishing_2025}
to 
machine learning \cite{xu_how_2019,morris_weisfeiler_2019,zhang_beyond_2024,gai_homomorphism_2025,cerny_caterpillar_2025} and
quantum information theory \cite{mancinska_quantum_2020,kar_npa_2025,seppelt_quantum_2025}.\footnote{For further references, visit the \emph{Homomorphism Indistinguishability Zoo} at \href{https://tseppelt.github.io/homind-database}{tseppelt.github.io/homind-database}.}

The plenitude of such characterisations motivates studying how properties of the homomorphism indistinguishability relation $\equiv_{\mathcal{F}}$ are related to properties of the graph class $\mathcal{F}$ \cite{seppelt_logical_2024}.
Concretely, the computational complexity \cite{boker_complexity_2019,seppelt_algorithmic_2024,cerny_homomorphism_2025} and the distinguishing power \cite{roberson_oddomorphisms_2022} of $\equiv_{\mathcal{F}}$ have been studied from this perspective. 
This paper is concerned with connections between notions of sparsity of $\mathcal{F}$ and the distinguishing power of $\equiv_{\mathcal{F}}$.

\subsection{Weak Roberson Conjecture: Separation from Isomorphism}

A systematic study of the distinguishing power of $\equiv_{\mathcal{F}}$ was initiated by Roberson in \cite{roberson_oddomorphisms_2022}.
Arguably, the most fundamental question in this context is whether $\equiv_{\mathcal{F}}$ is the same as the isomorphism relation when $\mathcal{F}$ is some proper graph class.
For example, it is not hard to see~\cite{lovasz_operations_1967} that two graphs are homomorphism indistinguishable over all connected graphs if, and only if, they are isomorphic.
A more involved example is given by Dvo\v{r}\'ak~\cite{dvorak_recognizing_2010}, who showed that two graphs $G$ and $H$ are homomorphism indistinguishable over all $2$-degenerate graphs if, and only if, they are isomorphic.

In contrast, there exist, for every $k \geq 0$, non-isomorphic graphs $G$ and $H$ that are homomorphism indistinguishable over the class of graphs of treewidth $\leq k$.
This is a consequence of the seminal result of Cai, F\"urer, and Immerman~\cite{cai_optimal_1992} that the logic $\mathsf{C}^{k+1}$, and equivalently the $k$\nobreakdash-dimensional Weisfeiler--Leman algorithm, fails to distinguish all pairs of non-isomorphic graphs.
Similarly, the known characterisation in terms of quantum isomorphisms \cite{mancinska_quantum_2020,atserias_quantum_2019} implies that homomorphism indistinguishability over the class of planar graphs is weaker than the isomorphism relation \cite{arkhipov_extending_2012}.
Inspired by these two examples, Roberson proposed the following conjecture:

\begin{conjecture}[\protect{Roberson~\cite[Conjecture 5]{roberson_oddomorphisms_2022}}]\label{conj:weak-roberson}
    For every $k \geq 0$,
    there exist non-isomorphic graphs $G$ and $H$ that are homomorphism indistinguishable over all graphs of Hadwiger number~$\leq k$.
\end{conjecture}

Here, the \emph{Hadwiger number} of a graph $F$ is the largest integer $k \in \mathbb{N}$ such that $F$ contains the $k$-vertex complete graph $K_k$ as a minor.
In other words, \cref{conj:weak-roberson} asserts that, for every graph class $\mathcal{F}$ excluding some minor, there exist non-isomorphic graphs $G$ and $H$ such that $G \equiv_{\mathcal{F}} H$.
Thus, the conjecture vastly generalises the aforementioned examples of planar graphs and graphs of bounded treewidth.
We stress that, prior to this work, no further examples were known to support \cref{conj:weak-roberson}.\footnote{In \cite{roberson_oddomorphisms_2022}, a proof of \cref{conj:weak-roberson} for $k=4$ due to Richter, Roberson, and Thomassen was announced. However, this result is to our knowledge still unpublished. Our \cref{cor:k7minorfree} subsumes this case.}

As our first main result, we prove \cref{conj:weak-roberson} for all minor-closed graph classes that do not exhibit a certain ingredient of the Robertson--Seymour structure theorem \cite{robertson_graph_2004} for minor-free graphs.
By this seminal theorem, for every graph $H$, every $H$-minor-free graph can be obtained by clique-sums of graphs that can be almost embedded on a surface $\Sigma$ of bounded Euler genus. 
Intuitively, a graph is \emph{almost embeddable} on $\Sigma$ if it can be obtained from a  graph embeddable on $\Sigma$ by attaching bounded-pathwidth graphs (so-called \emph{vortices}) at a bounded number of faces and adding a bounded number of vertices with arbitrary neighbourhood (so-called \emph{apices}).
Our main theorem applies to every minor-closed graph class omitting vortices.

\begin{theorem}[label=thm:main-vortex-free-hadwiger]
    For every $k \geq 0$,
    there exist non-isomorphic graphs $G$ and $H$ that are homomorphism indistinguishable over all graphs of vortex-free Hadwiger number $\leq k$.
\end{theorem}

The parameter \emph{vortex-free Hadwiger number} was introduced by 
Thilikos and Wiederrecht~\cite{thilikos_killing_2022}, who also gave a characterisation of this parameter in terms of a parametric family of forbidden minors analogous to the grid minor theorem for treewidth \cite{robertson_graph_1986}. 
Just like the Hadwiger number, the vortex-free Hadwiger number is a minor-monotone graph parameter.
The most notable special case of  \cref{thm:main-vortex-free-hadwiger} is the class of graphs of bounded genus.
It implies that, for every $g \geq 0$, there exist non-isomorphic graphs $G$ and $H$ that are homomorphism indistinguishable over all graphs that can be embedded on an (orientable or non-orientable) surface of genus $\leq g$.
Another special case of \cref{thm:main-vortex-free-hadwiger} is given by the graph classes that exclude a single-crossing minor~\cite{RobertsonS91}.
Since $K_7$-minor-free graphs have bounded vortex-free Hadwiger number \cite[Section~1.3]{thilikos_killing_2022}, we obtain the following corollary, which is the current state-of-the-art for \cref{conj:weak-roberson}.

\begin{corollary}\label{cor:k7minorfree}
    There exist non-isomorphic graphs $G$ and $H$ that are homomorphism indistinguishable over all graphs of Hadwiger number $\leq 6$.
\end{corollary}

It is natural to ask what happens beyond classes excluding a minor such as classes excluding a topological minor and other well-established notions in the theory of sparse graphs \cite{nesetril_sparsity_2012,siebertz_generalized_2025}.
For example, it is known that classes of bounded degree yield homomorphism indistinguishability relations that are weaker than isomorphism \cite{roberson_oddomorphisms_2022}, and characterisations have recently been obtained in \cite{cai_vanishing_2025}.
Note that graphs of maximum degree~$3$ contain every graph as a minor.
On the other hand, Dvo\v{r}\'ak~\cite{dvorak_recognizing_2010} obtained a graph class $\mathcal{F}$ of bounded expansion for which $\equiv_{\mathcal{F}}$ is the same as isomorphism.
Additionally, the graph class $\mathcal{F}$ is easily seen to have bounded local treewidth (see \cite{eppstein_subgraph_1995,grohe_local_2003}) and bounded twin-width \cite{bonnet_twin-width_2020}.
As our second main result, we strengthen those negative results and show that variants of \cref{conj:weak-roberson} fail for many natural notions that go beyond excluding a minor or bounding the degree, such as excluding a topological minor.

\begin{theorem}[restate=thmKfive,label=thm:k5-top-minor]
    Two graphs are isomorphic if, and only if, they are homomorphism indistinguishable over all graphs excluding $K_5$ as a topological minor.
\end{theorem}

\begin{figure}
	\centering
	\begin{tikzpicture}[
		box/.style={text width=3.7cm, rounded corners, font=\small, align=center, anchor=center},
		good/.style={box, fill=green!60!brown!30},
		unknown/.style={box, fill=lipicsYellow!25},
		bad/.style={box, fill=red!20}
		]
		
		  \matrix (mat) [matrix of nodes, nodes=good, 
		  	column sep=1cm, 
		  	row sep=.5cm
		 ] 
		{
            && \node [bad] (expansion) {bounded expansion \cite{dvorak_recognizing_2010}};\\
			& &  \node [bad] (hajos) {excluded topological minor (\cref{thm:k5-top-minor})}; \\
			\node [bad] (adm) {bounded $\infty$\nobreakdash-admis\-sibility (\cref{lem:d3-star})}; &
			\node [bad] (ltw) {bounded local treewidth \cite{dvorak_recognizing_2010} (\cref{lem:d3-star})}; &
			\node [unknown] (hadwiger) {excluded minor (\cref{conj:weak-roberson})}; \\
            \node [good] (immersion) {excluded weak immersion \cite{jimenez_homomorphism_2026}};
			&
			\node [good] (genus) {bounded Euler genus  (\cref{thm:genus-strictly-refining})}; &
			\node [good] (vortexfree) {bounded vortex-free Hadwiger number  (\cref{thm:main-vortex-free-hadwiger})}; \\
			&
			\node [good] (planar) {planar \cite{atserias_quantum_2019,mancinska_quantum_2020,roberson_oddomorphisms_2022}}; 
			
			&
			\node [good] (treewidth) {bounded treewidth \cite{cai_optimal_1992,dvorak_recognizing_2010,dell_lovasz_2018,neuen_homomorphism-distinguishing_2024}};  \\
			\node [good] (degree) {bounded degree \cite{roberson_oddomorphisms_2022}}; &
			\node [good] (trees) {forests \cite{dvorak_recognizing_2010,dell_lovasz_2018,roberson_oddomorphisms_2022}}; & \\
		};

		\draw [->] (trees) -- (planar);
		\draw [->] (trees) -| (treewidth);
		\draw [->] (treewidth) -- (vortexfree);
		\draw (planar) edge [->] (genus);
		\draw (genus) edge [->] (vortexfree) ;
		\draw (vortexfree) edge [->] (hadwiger);

		\draw (degree) edge [->] (immersion);
        \draw (immersion)edge [->] node [midway, anchor=west] {\cite{dvorak_stronger_2012}} (adm);
		\draw [->] (adm) |- (hajos); 

	     \draw (degree.east) -- ++(.5,0)  -- ($(ltw.west)+(-0.5,0)$) edge [->] (ltw.west);
\draw (hadwiger) edge [->] (hajos);
        \draw (hajos) edge [->] (expansion);
		\draw [->] (genus) edge node [midway, anchor=west] {\cite{eppstein_subgraph_1995}} (ltw);
	\end{tikzpicture}
	\caption{Properties of a graph class $\mathcal{F}$ that do (green) or do not (red) guarantee that there exist non-isomorphic graphs $G$ and $H$ that are homomorphism indistinguishable over $\mathcal{F}$.}
	\label{fig:graph-classes}
\end{figure}
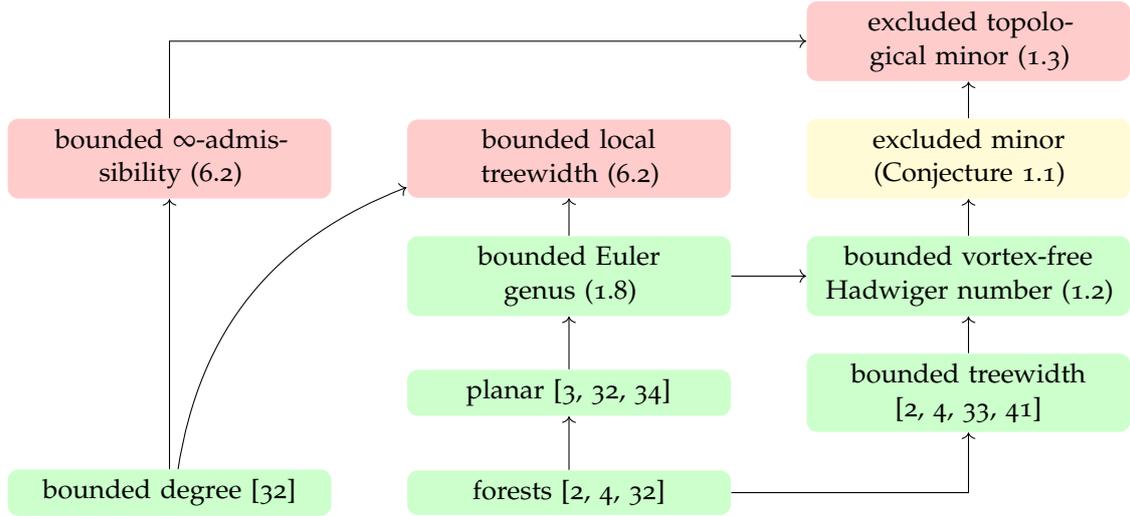

\cref{thm:k5-top-minor} gives further evidence for the special role that minor-closed graph classes play in the realm of homomorphism indistinguishability, e.g., as a demarcation line in \cref{fig:graph-classes}.
In fact, this is further highlighted by our approach to prove \cref{thm:k5-top-minor}:
We say a graph class $\mathcal{F}$ is \emph{closed under $2$\nobreakdash-sums with triangles} if, for every $F \in \mathcal{F}$ and every edge $e = vw \in E(F)$, we have $F \oplus_e K_3 \in \mathcal{F}$.
Here, $F \oplus_e K_3$ denotes the graph obtained from $F$ by adding a fresh vertex $x_e$ connected to exactly $v,w$ (i.e., we glue a triangle to the edge $e$).
Note that many natural graph classes are closed under $2$-sums with triangles, e.g., planar graphs, graphs of treewidth $\leq k$ for all $k \geq 2$, and the class of graphs excluding $K_5$ as a topological minor.
As our main technical insight, we prove that, for a class $\mathcal{F}$ closed under subgraphs and $2$-sums with triangles, we can recover homomorphism counts $\hom(F',\star )$ for every graph $F'$ that is a minor of some graph $ F \in \mathcal{F}$, given only the homomorphism counts $\hom(F,\star)$ from graphs $F \in \mathcal{F}$; see \cref{lem:contractors-for-minors}.

At this point, we directly obtain \cref{thm:k5-top-minor} by observing that the graphs of maximum degree~$3$ exclude $K_5$ as a topological minor, and every graph is a minor of a graph of maximum degree~$3$.
More concretely, our arguments yield a graph class $\mathcal{D}_3^*$ over which homomorphism indistinguishability is the same as isomorphism.
In addition to the fact that $\mathcal{D}_3^*$ excludes $K_5$ as a topological minor, it also has bounded $\infty$\nobreakdash-admissibility (see \cite{dvorak_constant-factor_2013,siebertz_generalized_2025}) and bounded local treewidth.

Finally, we note that in a concurrent and independent work, Jiménez, Moore, Quiroz, and Yoo \cite{jimenez_homomorphism_2026} show that for graph classes $\mathcal{F}$ that exclude a fixed graph as a weak immersion, the relation $\equiv_{\mathcal{F}}$ is not isomorphism. Graph classes with an excluded weak immersion sit between classes of bounded degree and classes of bounded $\infty$\nobreakdash-admissibility \cite{dvorak_stronger_2012}; see \cref{fig:graph-classes}.

\subsection{Strong Roberson Conjecture: Separating Homomorphism Indistinguishability Relations}

Once a  homomorphism indistinguishability relation $\equiv_{\mathcal{F}}$ is separated from isomorphism,
the question arises of how $\equiv_{\mathcal{F}}$ relates to other homomorphism indistinguishability relations~$\equiv_{\mathcal{F}'}$.
Clearly, if $\mathcal{F} \supseteq \mathcal{F}'$, then the relation $\equiv_{\mathcal{F}}$ refines $\equiv_{\mathcal{F}'}$.
When does the converse hold?

Answers to this question have been used to pinpoint the distinguishing power of graph isomorphism relaxations from finite model theory \cite{fluck_going_2024,adler_monotonicity_2024,neuen_homomorphism-distinguishing_2024,schindling_homomorphism_2025},
optimisation \cite{roberson_lasserre_2024},
and machine learning \cite{zhang_beyond_2024,gai_homomorphism_2025,cerny_caterpillar_2025}.

The central definition for answering this question is the following:
A graph class $\mathcal{F}$ is \emph{homomorphism distinguishing closed} \cite{roberson_oddomorphisms_2022}
if, for all $F' \not\in \mathcal{F}$, there exist graphs $G$ and $H$ that are homomorphism indistinguishable over $\mathcal{F}$ and satisfy $\hom(F', G) \neq \hom(F', H)$, i.e., $G $ and $H$
admit a different number of homomorphisms from $F'$.
In other words, homomorphism distinguishing closed graph classes $\mathcal{F}$ are maximal in the sense that, for every $F' \not\in \mathcal{F}$, the relation $\equiv_{\mathcal{F} \cup \{F'\}}$ is strictly finer than $\equiv_{\mathcal{F}}$.

This definition is significant since,
if $\mathcal{F}$ is homomorphism distinguishing closed,
then it holds that $\equiv_{\mathcal{F}}$ refines $\equiv_{\mathcal{F}'}$ if, and only if, $\mathcal{F} \supseteq \mathcal{F}'$.
Hence, in this case, the relations $\equiv_{\mathcal{F}}$  and  $\equiv_{\mathcal{F}'}$ 
can be compared simply by comparing the graph classes $\mathcal{F}$ and $\mathcal{F}'$.
The aforementioned articles \cite{fluck_going_2024,adler_monotonicity_2024,schindling_homomorphism_2025,roberson_lasserre_2024,zhang_beyond_2024,gai_homomorphism_2025,cerny_caterpillar_2025} made use of precisely this property.
Hence, it is desirable to exhibit homomorphism distinguishing closed graph classes.
To that end, Roberson~\cite{roberson_oddomorphisms_2022} proposed the following conjecture.

\begin{conjecture}[\protect{Roberson~\cite[Conjecture 4]{roberson_oddomorphisms_2022}}]\label{conj:strong-roberson}
    Every minor-closed and union-closed\footnote{A graph class $\mathcal{F}$ is \emph{union-closed} if $F_1, F_2 \in \mathcal{F}$ implies that their disjoint union $F_1 + F_2$ is in $\mathcal{F}$. It is not hard to see that every homomorphism distinguishing closed graph class is necessarily union-closed \cite{roberson_oddomorphisms_2022}.} graph class $\mathcal{F}$ is homomorphism distinguishing closed.
\end{conjecture}

Clearly, \cref{conj:strong-roberson} implies \cref{conj:weak-roberson}.
Note, however, that \cref{conj:weak-roberson} admits a certain monotonicity that  \cref{conj:strong-roberson}  lacks: 
A proof of \cref{conj:weak-roberson} for some $k \in \mathbb{N}$ immediately yields a proof for all $k' \leq k$.
In the case of \cref{conj:strong-roberson},
if a graph class $\mathcal{F}$ is  homomorphism distinguishing closed,
then its sub- and super-classes may or may not be homomorphism distinguishing closed, see \cite[Section~6]{seppelt_homomorphism_2024}.

 \cref{conj:strong-roberson} is known to be true for the class of planar graphs \cite{roberson_oddomorphisms_2022}, 
forests \cite{roberson_oddomorphisms_2022},
disjoint unions of paths (and cycles) \cite{roberson_oddomorphisms_2022},
graphs of treewidth $\leq k$ \cite{neuen_homomorphism-distinguishing_2024},
treedepth $\leq k$ \cite{fluck_going_2024},
pathwidth $\leq k$ \cite{seppelt_homomorphism_2024},
graphs with bounded-depth pebble forest covers \cite{adler_monotonicity_2024,schindling_homomorphism_2025}, and essentially finite graph classes \cite{seppelt_logical_2024}, see \cite[Section~6]{seppelt_homomorphism_2024}.
Most of these instances are motivated by the study of the expressive power of counting logic fragments \cite{neuen_homomorphism-distinguishing_2024,fluck_going_2024,adler_monotonicity_2024,seppelt_homomorphism_2024,schindling_homomorphism_2025}, and in fact the proofs for homomorphism distinguishing closedness often heavily rely on known characterisations.
This is a clear drawback towards verifying \cref{conj:strong-roberson}, since  not all minor-closed classes $\mathcal{F}$ are likely to admit natural characterisations for $\equiv_{\mathcal{F}}$.
In contrast, we obtain several new results in the realm of \cref{conj:strong-roberson} that do not rely on such characterisations and reprove some of the existing results using more direct and simpler arguments.

First, we resolve \cref{conj:strong-roberson} for all classes of forests.
In fact, our result for classes $\mathcal{F}$ of forests even holds under the weaker assumption that $\mathcal{F}$ is only closed under topological minors.

\begin{theorem}\label{thm:forest-main}
    \label{it:tree} Every union-closed class of forests closed under taking topological minors is homomorphism distinguishing closed.
\end{theorem}

For example, this generalises the known result that the class of forests of maximum degree~$d$ is homomorphism distinguishing closed \cite{roberson_oddomorphisms_2022}, which strengthened a previous result from \cite{grohe_homomorphism_2025}.
Moreover, as a direct corollary, we precisely determine the distinguishing power of the graph neural network architecture proposed in \cite[B.3]{cerny_caterpillar_2025} since it can be characterised by $\equiv_{\mathcal{F}}$ for certain classes $\mathcal{F}$ of forests.

Furthermore, we develop a combinatorial toolkit for proving that graph classes constructed via standard operations from sparsity theory such as taking clique-sums or adding apices are homomorphism distinguishing closed.
As a corollary, we extend the list of known homomorphism distinguishing closed graph classes by the following examples:

\begin{theorem}\label{thm:strong-roberson}
   The following graph classes are homomorphism distinguishing closed.
   \begin{enumerate}
       \item\label{it:cactus} the class of cactus graphs,
       \item\label{it:outerplanar} the class of outerplanar graphs, 
\item\label{it:k33} the class of $K_{3,3}$-minor-free graphs,
       \item\label{it:k5} the class of $K_{5}$-minor-free graphs,
       \item\label{it:apex-planar} for $k \geq 0$, the class of disjoint unions of $k$-apex planar graphs,
       \item\label{it:vertex-cover} for $k \geq 0$, the class of disjoint unions of graphs of vertex cover number $\leq k$, and,
       \item\label{it:fvs} for $k \geq 0$,  the class of disjoint unions of graphs of feedback vertex set number $\leq k$.
   \end{enumerate}
\end{theorem}

\cref{it:outerplanar} has an immediate corollary via results of \cite{roberson_lasserre_2024,dvorak_recognizing_2010,dell_lovasz_2018}
and thus answers a question of Roberson and Seppelt~\cite{roberson_lasserre_2024}.

\begin{corollary}\label{cor:lasserre}
	There exist graphs $G$ and $H$ such that the first-level Lasserre relaxation of the graph isomorphism integer program for $G$ and $H$ is feasible but $G$ and $H$ are distinguished by the $2$\nobreakdash-dimensional Weisfeiler--Leman algorithm.
\end{corollary}

Indeed,
it was shown in \cite{roberson_lasserre_2024} that the first-level Lasserre SDP relaxation of the graph isomorphism integer program for graphs $G$ and $H$ is feasible if, and only if, $G$ and $H$ are homomorphism indistinguishable over all outerplanar graphs.
Since the $2$-dimensional Weisfeiler--Leman algorithm distinguishes $G$ and $H$ if, and only if, there is a graph $F$ of treewidth $\leq 2$ such that $\hom(F, G) \neq \hom(F, H)$,
\Cref{cor:lasserre} follows from \cref{it:outerplanar} of \cref{thm:strong-roberson} as the complete bipartite graph $K_{2,3}$ has treewidth $2$ and is not outerplanar.

Finally, we consider
the class $\mathcal{E}_g$ of graphs of Euler genus $\leq g$
and the class $\mathcal{H}_k$ of graphs of vortex-free Hadwiger number $\leq k$
for $g, k \geq 0$.
Although we are unable to show that $\mathcal{E}_g$ and $\mathcal{H}_k$ are homomorphism distinguishing closed, 
we prove that their homomorphism indistinguishability relations are distinct. 

\begin{theorem}[restate=thmGenusStrictlyRefining,label=thm:genus-strictly-refining]
    For all $g \geq 0$,
    there exist graphs $G$, $H$ such that $G \equiv_{\mathcal{E}_g} H$
    and $G \not\equiv_{\mathcal{E}_{g+1}} H$.
\end{theorem}
\begin{theorem}[restate=thmHadStrictlyRefining,label=thm:had-strictly-refining]
    For all $k \geq 0$,
    there exist graphs $G$, $H$ such that $G \equiv_{\mathcal{H}_k} H$
    and $G \not\equiv_{\mathcal{H}_{k+1}} H$.
\end{theorem}

This yields two new provably infinite hierarchies of graph isomorphism relaxations.
Observe that \cref{thm:main-vortex-free-hadwiger} is a direct corollary of \cref{thm:had-strictly-refining}, since $G$, $H$ are in particular non-isomorphic.

\subsection{Analysing CFI Graphs via Oddomorphisms}

Having stated our main results, we proceed to describe the techniques underlying their proofs.
The central challenge that has to be overcome is that, for most of the graph classes~$\mathcal{F}$ considered in our results,
no characterisations of the homomorphism indistinguishability relations $\equiv_{\mathcal{F}}$ are known.
Since it is unreasonable to expect that the homomorphism indistinguishability relations of all minor-closed graph classes admit practical characterisations,
we do not attempt to establish such characterisations towards  \cref{conj:strong-roberson,conj:weak-roberson} but work directly with homomorphism counts.

Spelled out, proving that $\mathcal{F}$ is homomorphism distinguishing closed amounts to constructing, for every graph $G \not\in \mathcal{F}$,
two graphs $G_0$ and $G_1$ such that $G_0 \equiv_{\mathcal{F}} G_1$ and $\hom(G, G_0) \neq \hom(G, G_1)$.
The main source for such highly similar graphs $G_0$ and $G_1$ is the CFI construction of Cai, F\"urer, and Immerman~\cite{cai_optimal_1992},
who devised it to show that the $k$\nobreakdash-dimensional Weisfeiler--Leman algorithm does not distinguish all non-isomorphic graphs.
Given a connected graph $G$, the CFI construction yields two non-isomorphic graphs, the \emph{even CFI graph} $G_0$ and the \emph{odd CFI graph} $G_1$.
Originally,  Cai, F\"urer, and Immerman~\cite{cai_optimal_1992} showed that, if $G$ is sufficiently connected, then $G_0$ and $G_1$ are not distinguished by the $k$\nobreakdash-dimensional Weisfeiler--Leman algorithm.
Subsequently, this argument was refined \cite{dawar_power_2007,atserias_power_2007,AtseriasBD09,neuen_homomorphism-distinguishing_2024}
to the statement that, if the treewidth of $G$ is greater than $k$, 
then $G_0$ and $G_1$ are not distinguished by the $k$-dimensional Weisfeiler--Leman algorithm, and equivalently \cite{dvorak_recognizing_2010,dell_lovasz_2018}, they are homomorphism indistinguishable over all graphs of treewidth at most $k$.
Furthermore, it was shown in \cite{arkhipov_extending_2012,atserias_quantum_2019,mancinska_quantum_2020,roberson_oddomorphisms_2022} that, if $G$ is non-planar, then $G_0$ and $G_1$ are quantum isomorphic, i.e., homomorphism indistinguishable over all planar graphs.

These results already indicate that CFI graphs are a rather universal construction in light of \cref{conj:strong-roberson,conj:weak-roberson}.
Roberson~\cite{roberson_oddomorphisms_2022} formalised this observation by giving a combinatorial criterion for the CFI graphs $G_0$ and $G_1$ of some base graph $G$ to be homomorphism indistinguishable over a graph class $\mathcal{F}$.
More precisely,
Roberson~\cite[Theorem 3.13]{roberson_oddomorphisms_2022} 
showed that a graph~$F$ and a connected graph~$G$  satisfy that $\hom(F, G_0) \neq \hom(F, G_1)$ if, and only if, there exists a \emph{weak oddomorphism} $\phi \colon F\to G$.
A weak oddomorphism is a homomorphism satisfying certain parity constraints, see \cref{def:oddomorphism,fig:oddomorphism}.
Thus, in order to show that $\mathcal{F}$ is homomorphism distinguishing closed, it suffices to show that $\mathcal{F}$ is \emph{closed under weak oddomorphisms}, i.e., if $\phi \colon  F \to G$ is a weak oddomorphism and $F \in \mathcal{F}$, then $G \in \mathcal{F}$; see \cref{thm:rob62}.

In particular, \cref{conj:strong-roberson} would follow from the assertion that, for all graphs~$F$ and~$G$, if $G$ is connected and $\phi \colon F \to G$ is an oddomorphism, then $G$ is a minor of $F$ \cite[Question~5]{roberson_oddomorphisms_2022}. 
However, Kristj\'ansson~\cite{kristjansson_oddomorphisms_2026} recently exhibited a counterexample to this assertion.
We remark that for \cref{conj:weak-roberson}, it suffices to show that there exists an unbounded function $f \colon \mathbb{N} \to \mathbb{N}$ such that, for all $n \in \mathbb{N}$, if $F$ admits an oddomorphism to $K_n$, then $F$ contains $K_{f(n)}$ as a minor \cite[Question~7]{roberson_oddomorphisms_2022}. It remains open whether such a function $f$ exists.

Still, many natural graph classes are known to be closed under weak oddomorphisms, such as the class of planar graphs \cite{roberson_oddomorphisms_2022},
the class of graphs of maximum degree $\leq d$ \cite{roberson_oddomorphisms_2022},
treewidth $\leq k$ \cite{neuen_homomorphism-distinguishing_2024},
treedepth $\leq k$ \cite{fluck_going_2024}, and
pathwidth $\leq k$ \cite{seppelt_homomorphism_2024}.
These results were established mostly via a model-theoretic or algebraic analysis of CFI graphs and using characterisations of what it means for two graphs to be homomorphism indistinguishable over the respective graph classes.
The technical contribution of this paper is a set of purely combinatorial 
tools for establishing that a graph class is closed under weak oddomorphisms.

\begin{lemma}\label{lem:oddo-overview}
    Let $\mathcal{F}$ be closed under weak oddomorphisms and taking subgraphs.
    Let $d \geq 0$. 
    The following graph classes are closed under weak oddomorphisms:
    \begin{enumerate}
        \item\label{it:overview-oddo1} the class of graphs of deletion distance $\leq d$ to $\mathcal{F}$, 
        \item\label{it:overview-oddo2}  the class of graphs of elimination distance $\leq d$ to $\mathcal{F}$,
        \item\label{it:overview-oddo3}  the closure $\mathcal{F}^{\oplus 1}$ of $\mathcal{F}$ under $1$-sums, and, 
        \item\label{it:overview-oddo4}  if $\mathcal{F}$ is minor-closed,
        the closure $\mathcal{F}^{\oplus 2}$ of $\mathcal{F}$ under $2$\nobreakdash-sums.
    \end{enumerate}
\end{lemma}

Several of our main results follow as corollaries from this lemma.
For example, \Cref{it:overview-oddo1} yields that the class of disjoint unions of $k$-apex planar graphs is homomorphism distinguishing closed, and \Cref{it:overview-oddo2} gives a self-contained combinatorial proof of the fact that the class of graphs of treedepth $\leq k$ is homomorphism distinguishing closed \cite{fluck_going_2024}.
Additionally, we obtain in \cref{thm:oddo-pw,thm:oddo-tw} alternative proofs (which do not rely on model-theoretic characterizations in contrast to \cite{neuen_homomorphism-distinguishing_2024,seppelt_homomorphism_2024}) that the class of graphs of treewitdh (pathwidth) $\leq k$ is homomorphism distinguishing closed.

A crucial tool in graph minor theory are clique-sums, i.e., the operation of gluing together two graphs at cliques.
A key technical result of this paper is \cref{lem:separators}, which allows us to deal with weak oddomorphisms in light of clique-sums.
Intuitively, it asserts that, if $F_1 \oplus F_2 \to G$ is a weak oddomorphism from a clique-sum $F_1 \oplus F_2$ into a sufficiently connected graph $G$,
then there exists a minor $F'$ of $F_1$ or $F_2$ admitting a weak oddomorphism $F' \to G$.
\cref{it:overview-oddo3,it:overview-oddo4} of \cref{lem:oddo-overview} follow from this observation.
For higher-arity clique-sums, we give an approximate result in \cref{lem:oddo-to-kn-clique-sum} strong enough to yield \cref{thm:main-vortex-free-hadwiger,thm:genus-strictly-refining,thm:had-strictly-refining}.

By studying the combinatorial properties of oddomorphisms, we also contribute to the study of CFI graphs.
This line of work has recently received much attention \cite{lichter_separating_2021,grohe_compressing_2025} and enjoys connections to various other areas such as proof complexity \cite{de_rezende_truly_2025}, counting complexity \cite{Curticapean24},
and algebraic complexity \cite{dawar_symmetric_2020,dawar_symmetric_2025,dwivedi_lower_2026}. 	
	
	\section{Preliminaries}
	\label{sec:preliminaries}
	\subsection{Graphs}

We use standard graph notation (see, e.g., \cite{diestel_graph_2025}).
All graphs in this article are finite, undirected, loopless, and without multiple edges.
We write $V(G)$ and $E(G)$ for the vertex and edge set of a graph $G$, respectively.
For $v \in V(G)$, we write $N_G(v)$ for the set of neighbours of $v$.
Also, for $X \subseteq V(G)$ we write $G[X]$ for the subgraph of $G$ induced by $X$, and $G - X \coloneqq G[V(G) \setminus X]$.
For $k \geq 1$,
a graph $G$ is \emph{$k$-connected} if $|V(G)| > k$ and the graph $G - S$ is connected for every subset $S \subseteq V(G)$ such that $|S| < k$.

A \emph{tree decomposition} of a graph $G$ is a pair $(T, \beta)$ of a tree $T$ and a map $\beta \colon V(T) \to 2^{V(G)}$ such that
\begin{enumerate}
    \item for every edge $uv \in E(G)$, there exists a $t \in V(T)$ such that $u,v \in \beta(t)$, and
    \item for every $v \in V(G)$, the subgraph of $T$ induced by the vertices $t \in V(T)$ such that $v \in \beta(t)$ is non-empty and connected.
\end{enumerate}
The \emph{width} of $(T, \beta)$ is $\max_{t\in V(T)} |\beta(t)| - 1$.
The \emph{treewidth} $\tw(G)$ of $G$ is the minimum width of a tree decomposition of~$G$.
For $tt' \in E(T)$, we define $\sigma(t,t') \coloneqq |\beta(t) \cap \beta(t')|$ to be the \emph{adhesion} of $tt'$.
The \emph{adhesion} of $(T, \beta)$ is $\max_{tt' \in E(T)} \sigma(t,t')$.

Let $G_1$ and $G_2$ be two graphs whose vertex sets intersect in the set $S$.
Suppose that $G_1[S]$ and $G_2[S]$ are cliques.
We write $G_1 \oplus_S G_2$ for any graph obtained from the union of $G_1$ and $G_2$ by possibly deleting edges between vertices from $S$.
Such a graph is called a \emph{clique-sum} or \emph{$|S|$-sum} of $G_1$ and~$G_2$.
Note that a $0$-sum is merely a disjoint union.

For a graph class $\mathcal{C}$, we write $\mathcal{C}^{\oplus s}$ for the closure of $\mathcal{C}$ under $(\leq s)$-sums.
This can also be formalised via tree decompositions.
For $X \subseteq V(G)$, 
we define the \emph{torso of $G$ on~$X$} to be the graph $G\llbracket X \rrbracket$ with vertex set $V(G\llbracket X \rrbracket) \coloneqq X$ and edge set $E(G\llbracket X \rrbracket)$ given by
\begin{align*}
      \{vw \in E(G) \mid v,w \in X\} 
     \cup \{xy
    \mid x,y \in N_G(C),\ C \text{ connected component of } G - X\}.
\end{align*}

\begin{fact}\label{def:clique-sum}
    Let $\mathcal{C}$ be a class of graphs and $s \geq 0$.
   Then $\mathcal{C}^{\oplus s}$ is the class of graphs $G$ such that there is a tree decomposition $(T,\beta)$ of $G$ such that
    \begin{enumerate}
        \item $\sigma(t,t') \leq s$ for every $tt' \in E(T)$, and
        \item $G\llbracket \beta(t)\rrbracket \in \mathcal{C}$ for every $t \in V(T)$.
    \end{enumerate}
\end{fact}

\begin{fact}[label=lem:cliquesum-minor-closedness]
    Let $\mathcal{F}$ be a graph class and $k \geq 0$.
    If $\mathcal{F}$ is minor-closed, then so is $\mathcal{F}^{\oplus k}$.
\end{fact}
\begin{proof}The proof is by structural induction on the graphs in $\mathcal{F}^{\oplus k}$.
	The base case holds by assumption.
	Now assume that $M$ is a minor of $F_1 \oplus_S F_2$ for $F_1, F_2 \in \mathcal{F}^{\oplus k}$ of lesser complexity.
	It follows from \cite[Lemma~4.14.2]{roberson_lasserre_2024} that $M$ can be written as the $k$-sum of two graphs $M_1$ and $M_2$ such that $M_1$ is a minor of $F_1$ and $M_2$ is a minor of $F_2$.
	By the induction hypothesis, $M_1, M_2 \in \mathcal{F}^{\oplus k}$ and hence $M \in \mathcal{F}^{\oplus k}$.
\end{proof}

\subsection{Homomorphism Indistinguishability}

A \emph{homomorphism} from a graph $F$ to a graph $G$ is a map $\phi \colon V(F) \to V(G)$ such that $\phi(u)\phi(v)\in E(G)$ for all $uv \in E(F)$.
Write $\hom(F, G)$ for the number of homomorphisms $F \to G$.
Two graphs $G$ and $H$ are \emph{homomorphism indistinguishable} over a graph class $\mathcal{F}$, in symbols $G \equiv_{\mathcal{F}} H$, if $\hom(F, G) = \hom(F, H)$ for every $F \in \mathcal{F}$.
The \emph{homomorphism distinguishing closure} \cite{roberson_oddomorphisms_2022} of a graph class $\mathcal{F}$ is the graph class $\cl(\mathcal{F})$ containing all graphs $F$ such that, for all graphs $G$ and $H$, if $G \equiv_{\mathcal{F}} H$, then $\hom(F, G) = \hom(F, H)$.
In other words, $\cl(\mathcal{F})$ is the maximal graph class such that $\equiv_{\mathcal{F}}$ and $\equiv_{\cl(\mathcal{F})}$ coincide.
A graph class $\mathcal{F}$ is \emph{homomorphism distinguishing closed} if $\mathcal{F} = \cl(\mathcal{F})$.
We note that $\cl(\cdot)$ is a closure operator in the sense that
\begin{align}
    \cl(\mathcal{F}) &= \cl(\cl(\mathcal{F})),\label{eq:double-closure} \\
    \cl(\mathcal{F}) & \subseteq \cl(\mathcal{F}')\label{eq:subset-closure}
\end{align}
for all graph classes $\mathcal{F} \subseteq \mathcal{F}'$.
We refer to \cite{seppelt_homomorphism_2024} for further background.

\subsection{Oddomorphisms}

As already indicated above, a key notion underlying many of our results is that of an \emph{oddomorphism} introduced by Roberson~\cite{roberson_oddomorphisms_2022}
to give a combinatorial criterion for the homomorphism indistinguishability of CFI graphs.
We recall the essential definitions and properties; see \cref{fig:oddomorphism} for an example.

\begin{figure}
    	\centering
    	
    	\tikzset{vertex/.style={draw,circle,fill=lipicsBulletGray,line width=1pt, inner sep=2.5},
    		odd/.style={vertex ,fill=lipicsGray,line width=1.5pt},
    		even/.style={vertex, fill=none, line width=1.5pt},
    		every node/.style={anchor=center}}
    	
    	\begin{tikzpicture}
    		
    		\node [vertex,rectangle] (g1) {};
    		\node [vertex, rectangle, right of=g1, xshift=1cm] (g2) {};
    		\node [vertex, rectangle, right of=g2, xshift=1cm] (g3) {};
    		\node [vertex,rectangle, right of=g3, xshift=1cm] (g4) {};
    		\node [vertex,rectangle, right of=g4, xshift=1cm] (g5) {};
    		
    		\draw [thick] (g1) -- (g2) -- (g3) -- (g4) -- (g5);

    		\node [odd, above of=g1, yshift=2.5cm] (f1) {}; 
    		\node [even, above of=g2, yshift=3cm] (f2)  {}; 
    		\node [even, above of=g1, yshift=1.5cm] (f3) {}; 
    		\node [even, above of=g2, yshift=2cm] (f4) {};
    		\node [even, above of=g1, yshift=.5cm] (f5) {}; 
    		\node [odd, above of=g2, yshift=1cm] (f6) {};
    		
    		\draw [ thick] (f1) -- (f2) -- (f3) -- (f4) -- (f5) -- (f6);
    		
    		\node [even, above of=g3, yshift=1.5cm] (f7) {};
    		
    		\node [odd, above of=g3, yshift=2.5cm] (f9) {};
    		\draw [ thick] (f6) -- (f7) -- (f4) -- (f9);
    		\node [even, above of=g4, yshift=2cm] (f8) {};
    		\draw [ thick] (f7) -- (f8) -- (f9);
    		
    		\node [odd, above of=g4, yshift=1cm] (f10) {};
    		
    		\draw [thick] (f7) -- (f10);
    		
    		\node [odd, above of=g5, yshift=2cm] (f11) {};
    		\node [odd, above of=g5, yshift=1cm] (f12) {};
    		\node [odd, above of=g5, yshift=0] (f13) {};
    		
    		\draw [thick] (f10) -- (f11);
    		\draw [thick] (f10) -- (f12);
    		\draw [thick] (f10) -- (f13);
    		
    		\draw [dashed, thin] (f1) -- (f3) -- (f5) edge [->] (g1);
    		\draw [dashed, thin] (f2) -- (f4) -- (f6) edge [->] (g2);
    		\draw [dashed, thin] (f9) -- (f7) edge [->] (g3);
    		\draw [dashed, thin] (f8) -- (f10) edge [->] (g4);
    		\draw [dashed, thin] (f11) -- (f12) -- (f13) edge [->] (g5);
    	\end{tikzpicture}
    	\caption[An oddomorphism]{An oddomorphism (\begin{tikzpicture}
    			\node (a) [inner sep=0] {};
    			\node (b) [right of=a,inner sep=0] {};
    			\draw [dashed, thin,->] (a) -- (b);
    		\end{tikzpicture}) to the $5$-vertex path. Even vertices are depicted as~\begin{tikzpicture}
    			\node [even] {};
    		\end{tikzpicture} and odd vertices as \begin{tikzpicture}
    			\node [odd] {};
    		\end{tikzpicture}.}
    	\label{fig:oddomorphism}
    \end{figure}

    \begin{definition}[{\cite[Definition~3.9]{roberson_oddomorphisms_2022}}] \label{def:oddomorphism}
    	Let $F$ and $G$ be graphs and $\phi \colon F\to G$ a homomorphism.
    	A vertex $a \in V(F)$ is \emph{odd/even with respect to $\phi$} if $\lvert N_F(a) \cap \phi^{-1}(v) \rvert$ is odd/even for every $v \in N_G(\phi(a))$.
    	The homomorphism~$\phi$ is an \emph{oddomorphism} if
    	\begin{enumerate}
    		\item every vertex of $F$ is even or odd with respect to $\phi$,
    		\item for every $v\in V(G)$, the set $\phi^{-1}(v)$ contains an odd number of odd vertices.
    	\end{enumerate}
    	The homomorphism~$\phi \colon F\to G$ is a \emph{weak oddomorphism} if there is a subgraph $F'$ of $F$ such that $\phi|_{F'}$ is an oddomorphism from $F'$ to~$G$.
    \end{definition}
    
    If a vertex~$a \in V(F)$ is odd or even with respect to $\phi$, it is referred to as \emph{$\phi$-odd} or \emph{$\phi$-even}, respectively.
    The sets $\phi^{-1}(v) \subseteq V(F)$ for $v \in V(G)$ are called the \emph{fibres}\index{fibre} of $\phi$.

    \begin{example}\label{ex:identity}
        For every graph $F$, the identity map $\id \colon F \to F$ is an oddomorphism.
    \end{example}
    
    \begin{observation}\label{obs:surjective}
    	Every weak oddomorphism $\phi \colon F\to G$ is surjective on edges and vertices.
    \end{observation}

    Our motivation for studying (weak) oddomorphisms is the following \cref{thm:rob62}.
    We say a graph class $\mathcal{F}$ is \emph{closed under (weak) oddomorphisms} if, for every $F \in \mathcal{F}$ and (weak) oddomorphism $\phi \colon F \to G$, it holds that $G \in \mathcal{F}$.
    Also, a graph class $\mathcal{F}$ is \emph{componental} if, for all graphs $F_1$, $F_2$, it holds that $F_1 + F_2 \in \mathcal{F}$ if, and only if, $F_1, F_2 \in \mathcal{F}$.

    \begin{theorem}[{\cite[Theorem~6.2]{roberson_oddomorphisms_2022}}]\label{thm:rob62}
        Let $\mathcal{F}$ be a componental graph class.
        If $\mathcal{F}$ is closed under weak oddomorphisms, then $\mathcal{F}$ is homomorphism distinguishing closed.
    \end{theorem}

    We recall the following lemma, which implies via \cref{thm:rob62} that the class of planar graphs, the class of graphs of treewidth $\leq k$,
    and the class of graphs $F$ of maximum degree $\Delta(F) \leq d$ are homomorphism distinguishing closed.
    While the first two statements are involved and rely on characterisations \cite{mancinska_quantum_2020,dvorak_recognizing_2010,dell_lovasz_2018} of homomorphism indistinguishability relations, 
    the third statement is rather immediate from \cref{def:oddomorphism}.
    For \cref{sec:separators}, the second assertion is required only for $\tw(F) \leq 2$,
    for which we give a self-contained proof in \cref{thm:tw1,thm:tw2}.
    Moreover, in \cref{thm:oddo-tw}, we give a self-contained proof of \cref{it:tw-oddo}.

    \begin{lemma}[\protect{\cite[Lemmas 4.1 and 8.2]{roberson_oddomorphisms_2022}, \cite[Corollary~13]{neuen_homomorphism-distinguishing_2024}}]\label{lem:oddo-planar-degree}
        The following holds for every weak oddomorphism $\varphi\colon F \to G$.
        \begin{enumerate}[label = (\alph*)]
            \item\label{it:planar-oddo} If $F$ is planar, then $G$ is planar.
            \item\label{it:tw-oddo} $\tw(F) \geq \tw(G)$.
            \item $\Delta(F) \geq \Delta(G)$.
        \end{enumerate}
    \end{lemma}

    We recall the following properties of oddomorphisms.

    \begin{lemma}[{\cite[Lemmas~5.2, 5.6, and Theorem~3.13]{roberson_oddomorphisms_2022}}]
        \label{lem:oddo-connected}\label{lem:oddo-minors}
        Let $\phi \colon F \to G$ be a weak oddomorphism.
        \begin{enumerate}[label = (\alph*)]
            \item\label{item:oddo-minors-1} For every subgraph $G'$ of $G$,
                    the induced map $\phi|_{\phi^{-1}(G')} \colon \allowbreak F[\phi^{-1}(G')] \to G'$ is a weak oddomorphism.
                    If $\phi$ is an oddomorphism, then so is $\phi|_{\phi^{-1}(G')}$.
            \item\label{item:oddo-minors-2} For every connected subgraph $G'$ of $G$,
                    there exists a connected subgraph $F'$ of $F[\phi^{-1}(G')]$ admitting an oddomorphism $F' \to G'$.
            \item\label{item:oddo-minors-3} 
            For every minor $G'$ of $G$, there exists a minor $F'$ of $F$ admitting an oddomorphism $F' \to G'$.
            If $G'$ has fewer vertices than $G$,
            then $F'$ has fewer vertices than $F$.
        \end{enumerate}
    \end{lemma}

    In order to simplify some subsequent arguments, we make the following observation.

    \begin{observation}\label{lem:oddo-woddo}
        For every graph class $\mathcal{F}$ closed under taking subgraphs,
        the following are equivalent:
        \begin{enumerate}
            \item $\mathcal{F}$ is closed under weak oddomorphisms,
            \item $\mathcal{F}$ is closed under oddomorphisms.
        \end{enumerate}
    \end{observation}

    \begin{proof}
        Since every oddomorphism is a weak oddomorphism, the first item implies the second.
        For the converse direction, let $F \to G$ be a weak oddomorphism and $F \in \mathcal{F}$.
        Then there exists an oddomorphism $F' \to G$ from some subgraph $F'$ of $F$ by \cref{def:oddomorphism}.
        In particular, $F' \in \mathcal{F}$ and thus $G \in \mathcal{F}$.
    \end{proof}

    Finally, we shall use the following \cref{lem:strictness-via-oddo} to separate homomorphism indistinguishability relations.

    \begin{lemma}[{\cite[Theorem~3.13 and Lemma~3.14]{roberson_oddomorphisms_2022}}]
        \label{lem:strictness-via-oddo}
        Let $\mathcal{F}$ and $\mathcal{F'}$ be graph classes.
        If there is a connected graph $F' \in \mathcal{F'}$ such that no graph $F \in \mathcal{F}$ admits a weak oddomorphism to $F'$,
        then there exist graphs $G$, $H$ such that $G \equiv_{\mathcal{F}} H$, but $G \not\equiv_{\mathcal{F'}} H$.
    \end{lemma}

	\section{Deletion and Elimination Distance}
	In this section, we show that, if a graph class $\mathcal{F}$ is closed under weak oddomorphisms, then so are the classes of graphs of bounded deletion and elimination distance to $\mathcal{F}$.
As a corollary, we show that the classes of $k$-apex planar graphs and 
of graphs of vertex cover or feedback vertex set number $\leq k$ are homomorphism distinguishing closed (\cref{it:apex-planar,it:vertex-cover,it:fvs} of \cref{thm:strong-roberson}).
We also give a self-contained combinatorial argument for the fact that the classes of graphs of treedepth $\leq k$ are homomorphism distinguishing closed \cite{fluck_going_2024}.
\label{sec:deletion-distance}

    \begin{definition}[\cite{bulian_graph_2016}]
		Let $\mathcal{F}$ be a graph class.
		The \emph{deletion distance $\dd_{\mathcal{F}}(F)$ to $\mathcal{F}$} of a graph $F$ is defined as the least number $k$ such that there exist vertices $v_1, \dots, v_k \in V(F)$ such that $F - \{v_1, \dots, v_k\} \in \mathcal{F}$.
	\end{definition}

    For example, the \emph{vertex cover number} of a graph is equal to its deletion distance to the class of edge-less graphs.
    The \emph{feedback vertex set number} of a graph is equal to its deletion distance to forests.

    \begin{lemma}\label{lem:deletion-distance-oddo}
        Let $\mathcal{F}$ be closed under weak oddomorphisms and  taking subgraphs.
        If $F \to G$ is a weak oddomorphism,
        then $\dd_{\mathcal{F}}(F) \geq \dd_{\mathcal{F}}(G)$.
    \end{lemma}
    \begin{proof}
        Let $S \subseteq V(F)$ be such that $F - S \in \mathcal{F}$.
        Let $\phi \colon F \to G$ be a weak oddomorphism.
        Let $T \coloneqq \phi(S)$.
        Clearly, $|T| \leq |S|$.
        We show that $G - T \in \mathcal{F}$.
        By \cref{lem:oddo-connected},
        there exists a weak oddomorphism $F' \to G- T$ for some subgraph $F'$ of $F - S$.
        Since $\mathcal{F}$ is closed under taking subgraphs, it holds that $F' \in \mathcal{F}$ and hence $G - T \in \mathcal{F}$.
    \end{proof}

    \cref{lem:deletion-distance-oddo} has the following corollary,
    which implies \cref{it:apex-planar,it:vertex-cover,it:fvs} of \cref{thm:strong-roberson} via \cref{lem:oddo-planar-degree}.

    \begin{corollary}
        Let $\mathcal{F}$ be closed under weak oddomorphisms and taking subgraphs.
        For every $k \geq 0$,
        the class of graphs whose connected components have deletion distance at most~$k$ to $\mathcal{F}$ is closed under weak oddomorphisms.
        In particular, it is homomorphism distinguishing closed.
    \end{corollary}
    \begin{proof}
        Let $F \to G$ be a weak oddomorphism and suppose that $F$ is the disjoint union of graphs whose deletion distance to $\mathcal{F}$ is at most~$k$.
        Towards concluding that $G$ is the disjoint union of graphs whose deletion distance to $\mathcal{F}$ is at most $k$,
        we may suppose, by \cref{lem:oddo-connected},
        that $F$ and $G$ are connected.
        By \cref{lem:deletion-distance-oddo}, 
        $\dd_{\mathcal{F}}(G) \leq \dd_{\mathcal{F}}(F) \leq k$, as desired.
        The final assertion follows from \cref{thm:rob62}.
    \end{proof}

    Next, we consider the more general notion of elimination distance, which is inspired by treedepth \cite{nesetril_tree-depth_2006}.
    See \cite{jansen_vertex_2021} for computational aspects.
    
	\begin{definition}[\cite{bulian_graph_2016}]
	The \emph{elimination distance} to a graph class $\mathcal{F}$ of a graph $F$ is defined~as
		\[
		\ed_{\mathcal{F}}(F) \coloneqq	\begin{cases}
				0, & \text{if } F \in \mathcal{F}, \\
				1 + \min_{v \in V(F)} \ed_{\mathcal{F}}(F -v), & \text{if } F \text{ is connected},\\
				\max_i \ed_{\mathcal{F}}(F_i), & \text{if } F = \sum F_i \text{ is disconnected}.
			\end{cases}
		\]
	\end{definition}

    The \emph{treedepth} \cite{nesetril_tree-depth_2006} of a graph is equal to its elimination distance to the graph class containing only the empty graph.
    We recall the following lemma.

    \begin{lemma}[{\cite[p.\ 154]{bulian_fixed-parameter_2017}}]\label{lem:elimdistance-monotone}
        Let $\mathcal{F}$ be closed under taking subgraphs.
        If $F'$ is a subgraph of~$F$, then
        $\ed_{\mathcal{F}}(F') \leq \ed_{\mathcal{F}}(F)$.
    \end{lemma}

    The following \cref{thm:elimination-distance} generalises a result from \cite{fluck_going_2024} asserting that the class of graphs of treedepth $\leq q$ is closed under weak oddomorphisms.
    Furthermore, in contrast to \cite{fluck_going_2024}, our proof of \cref{thm:elimination-distance} is purely combinatorial and does not rely on Ehrenfeucht--Fraïssé or graph searching games.

	\begin{theorem}\label{thm:elimination-distance}
		Let $d \geq 0$
		and $\mathcal{F}$ be closed under taking subgraphs.
		If $\mathcal{F}$ is closed under weak oddomorphisms,
		then so is the class of all graphs $F$ with $\ed_{\mathcal{F}}(F) \leq d$.
	\end{theorem}
	\begin{proof}
		The proof is by induction on $d$.
		First consider the case $d = 0$.
		Let $F$ be such that $\ed_{\mathcal{F}}(F) = 0$
		and let $F \to G$ be a weak oddomorphism.
		Let $G'$ be a connected component of $G$.
		By \cref{lem:oddo-connected},
		there exists a connected subgraph $F'$ of $F$
		admitting an oddomorphism $F' \to G'$.
		By assumption, $F' \in \mathcal{F}$
		and hence $G' \in \mathcal{F}$.
		Since the elimination distance to $\mathcal{F}$
		does not increase under taking disjoint unions,
		it follows that $G \in \mathcal{F}$.
		
		For the inductive step, let $\phi \colon F \to G$ be a weak oddomorphism and first suppose that $F$ is connected and that $|V(F)| \geq 2$.
		Since surjective homomorphisms map connected graphs to connected graphs,
		it follows that $G$ is connected.
		Without loss of generality, it may be supposed that $|V(G)| \geq 2$.

		It holds that $\ed_{\mathcal{F}}(F) = 1 + \min_{a \in V(F)} \ed_{\mathcal{F}}(F- a)$.
		Let $a \in V(F)$ be such that $\ed_{\mathcal{F}}(F) = 1 + \ed_{\mathcal{F}}(F-a)$.
		Define $G' \coloneqq G - \phi(a)$ and $F' \coloneqq F[\phi^{-1}(G')] =  F - \phi^{-1}(\phi(a))$.
		By \cref{lem:oddo-connected},
		the map $\phi' \colon F' \to G'$ induced by $\phi$ is a weak oddomorphism.
		Note that $F'$ is a subgraph of $F-a$.
		By \cref{lem:elimdistance-monotone}, $\ed_{\mathcal{F}}(F') \leq \ed_{\mathcal{F}}(F - a) = \ed_{\mathcal{F}}(F) -1$.
		By the inductive hypothesis,
		$\ed_{\mathcal{F}}(G') \leq \ed_{\mathcal{F}}(F')$.
		Combining the above,
		\[
		\ed_{\mathcal{F}}(G) \leq 1 + \ed_{\mathcal{F}}(G') 
		\leq 1 + \ed_{\mathcal{F}}(F') \leq \ed_{\mathcal{F}}(F).
		\]

		It remains to consider the case when $F$ is disconnected.
		Write $G = G_1 + \dots + G_n$ as a disjoint union of connected components.
		Note that it may be the case that $G$ is connected and consists of a single connected component.
		By \cref{lem:oddo-connected},
		for every $i \in [n]$,
		there exists a connected subgraph $F_i$ of $F$ admitting an oddomorphism $\phi_i \colon F_i \to G_i$.
		By \cref{lem:elimdistance-monotone}, $\ed_{\mathcal{F}}(F_i) \leq \ed_{\mathcal{F}}(F)$.
		Since $F_i$ is connected, the previous case applies.
		It follows that
		\[
			\ed_{\mathcal{F}}(G) =\max_{ i\in [n]} \ed_{\mathcal{F}}(G_i)
			\leq \max_{ i\in [n]} \ed_{\mathcal{F}}(F_i)
			\leq \ed_{\mathcal{F}}(F). \qedhere
		\]
	\end{proof}

	\section{Treewidth and Pathwidth}
	\label{sec:tree-dec}
	\subsection{Treewidth}
	Next, we consider graphs of small treewidth.
\Cref{lem:oddo-planar-degree}\ref{it:tw-oddo}  states that $\tw(G) \leq \tw(F)$ for every (weak) oddomorphism $\varphi\colon F \to G$.
However, the only known proof from \cite{neuen_homomorphism-distinguishing_2024} crucially exploits a characterization of homomorphism indistinguishability over the graphs of treewidth~$\leq k$ from finite model theory \cite{dvorak_recognizing_2010,dell_lovasz_2018}.
In the following, we give an alternative proof that avoids any tools from finite model theory, and instead rests on the duality between treewidth and so-called havens.
In fact, we prove the following slightly stronger statement.

\begin{theorem}
    \label{thm:oddo-tw}
    Let $\varphi\colon F \to G$ be an oddomorphism, and let $(T,\beta)$ be a tree decomposition of $F$.
    Then
    \[\tw(G) \leq \max_{t \in V(T)} \lvert \varphi(\beta(t)) \rvert - 1.\]
\end{theorem}

Towards the proof, we start by defining havens.

\begin{definition}[{\cite[802]{alon_separator_1990}}, {\cite[Section~1]{seymour_graph_1993}}]\label{def:haven}
    Let $G$ be a graph and $h \geq 1$.
    A \emph{haven of order~$h$} in $G$ is a map $\eta$ assigning to every $X \subseteq V(G)$ with $|X| < h$ a connected component $\eta(X)$ of $G - X$ such that, for every $X \subseteq Y \subseteq V(G)$ with $|Y| < h$, it holds that $\eta(Y) \subseteq \eta(X)$.
\end{definition}

The following theorem characterizes the treewidth of a graph via havens.

\begin{theorem}[{\cite[(1.7)]{alon_separator_1990}}, see also {\cite[(1.4)]{seymour_graph_1993}}]\label{thm:haven-duality}
    Let $G$ be a graph and $h \geq 1$.
    Then $\tw(G) \geq h-1$ if, and only if, $G$ admits a haven of order~$h$.
\end{theorem}

Also, we shall require the following lemma.
Let $F$ be a graph.
A pair $(A_1,A_2)$, where $A_1,A_2 \subseteq V(F)$, is a \emph{separation} if $A_1 \cup A_2 = V(F)$ and there is no edge from $A_1 \setminus A_2$ to $A_2 \setminus A_1$.

\begin{lemma}
    \label{lem:oddo-separation}
    Let $\varphi\colon F \to G$ be an oddomorphism, and let $(A_1,A_2)$ be a separation in $F$.
    Let $C$ be a connected component of $G - \varphi(A_1 \cap A_2)$.
    Then there is a unique $j \in \{1,2\}$ such that $\varphi^{-1}(v) \cap A_j$ contains an odd number of $\varphi$-odd vertices for every $v \in C$.
\end{lemma}

\begin{proof}
    Define $V_1 \coloneqq A_1 \setminus A_2$ and $V_2 \coloneqq A_2 \setminus A_1$.
    Consider some $v \in C$.
    Since $\varphi^{-1}(v) \subseteq V_1 \cup V_2$ and $\varphi^{-1}(v)$ contains an odd number of $\varphi$-odd vertices, we conclude that, among the two sets $\phi^{-1}(v) \cap V_1$ and $\phi^{-1}(v) \cap V_2$, one set needs to contain an odd number of $\varphi$-odd vertices, while the other set contains an even number of $\varphi$-odd vertices.
    Let $b_v \in \{1,2\}$ denote the index such that $\phi^{-1}(v) \cap V_{b_v}$ contains an odd number of $\varphi$-odd vertices.

    Now, consider a second vertex $w \in C$ such that $vw \in E(G)$.
    \begin{itemize}
        \item Every $\varphi$-odd vertex $a \in \phi^{-1}(v) \cap V_{b_v}$ has an odd number of neighbors in $\phi^{-1}(w)$, and all those neighbors need to lie in $V_{b_v}$ since they are no edges between $V_1$ and $V_2$. 
        \item Every $\varphi$-even vertex $a \in \phi^{-1}(v) \cap V_{b_v}$ has an even number of neighbors in $\phi^{-1}(w)$, and all those neighbors need to lie in $V_{b_v}$ since they are no edges between $V_1$ and $V_2$.
    \end{itemize}
    Together, the parity of the number of edges from $\phi^{-1}(v) \cap V_{b_v}$ to $\phi^{-1}(w) \cap V_{b_v}$ equals the parity of the number of $\varphi$-odd vertices in $\phi^{-1}(v) \cap V_{b_v}$.
    Swapping the roles of $v$ and $w$, we also obtain that the parity of the number of edges from $\phi^{-1}(v) \cap V_{b_v}$ to $\phi^{-1}(w) \cap V_{b_v}$ equals the parity of the number of $\varphi$-odd vertices in $\phi^{-1}(w) \cap V_{b_v}$.
    Hence, $\phi^{-1}(w) \cap V_{b_v}$ contains an odd number of $\varphi$-odd vertices, i.e., $b_v = b_w$.
    Since $C$ is connected, we obtain that $b_v = b_w$ for all $v,w \in C$.
\end{proof}

We are ready to prove \Cref{thm:oddo-tw}.

\begin{proof}[Proof of \Cref{thm:oddo-tw}]
    Let $h \coloneqq \max_{t \in V(T)} |\varphi(\beta(t))|$.
    We show that there is no heaven of order $h+1$ in $G$, which implies that $\tw(G) \leq h-1$ by \Cref{thm:haven-duality}. 
    Suppose towards a contradiction that there is a haven $\eta$ of order $h+1$ in $G$. 
    
    Consider an edge $t_1t_2 \in E(T)$ and let $S \coloneqq \beta(t_1) \cap \beta(t_2)$.
    Also let $T_1$ and $T_2$ denote the two subtrees obtained from $T$ after deleting the edge $t_1t_2$, such that $t_1 \in V(T_1)$ and $t_2 \in V(T_2)$.
    Also, let $A_j \coloneqq \bigcup_{t_j \in V(T_j)} \beta(t_j)$ for both $j \in \{1,2\}$.
    It is well-known that $(A_1,A_2)$ is a separation in $F$ (see, e.g., \cite[Lemma 12.3.1]{diestel_graph_2025}) and $S = A_1 \cap A_2$.

    Now, let $C \coloneqq \eta(\varphi(S))$, which is a connected component of $G - \varphi(S)$.
    Hence, by \Cref{lem:oddo-separation}, there is some a unique $j \in \{1,2\}$ such that $\varphi^{-1}(v) \cap A_j$ contains an odd number of $\varphi$-odd vertices for every $v \in C$.
    This means that $\varphi^{-1}(v) \cap A_{3-j}$ contains an even number of $\varphi$-odd vertices for every $v \in C$.
    We orient the edge $t_1t_2$ towards $t_j$.

    After doing the same for all edges of $T$, we obtain an orientation of $T$, which means that there exists a node $t \in V(T)$ which has only incoming edges (since $T$ is a tree).
    Let $s_1,\dots,s_\ell$ denote the neighbors of $t$, and $T_1,\dots,T_\ell$ denote the subtrees of $T - \{t\}$ such that $s_i \in V(T_i)$ for all $i \in [\ell]$.
    Also, let $V_i \coloneqq \bigcup_{t_i \in V(T_i)} \beta(t_i)$.
    
    Now, let $Y \coloneqq \varphi(\beta(t))$, and fix some $v \in \eta(Y)$.
    Consider some $i \in [\ell]$, and let $X \coloneqq \varphi(\beta(t) \cap \beta(s_i)) \subseteq Y$.
    Hence, $\eta(Y) \subseteq \eta(X)$ which implies that $v \in \eta(X)$.
    Since the edge $ts_i$ is oriented towards $t$, we conclude that $\varphi^{-1}(v) \cap V_i$ contains an even number of $\varphi$-odd vertices.
    Because $\varphi^{-1}(v) \subseteq V_1 \cup \dots \cup V_\ell$, we obtain that $\varphi^{-1}(v)$ contains an even number of $\varphi$-odd vertices, which is a contradiction.
\end{proof}
 	\subsection{Pathwidth}
	The argument of the previous subsection can also be adapted from treewidth to pathwidth.
It is known that $\pw(G) \leq \pw(F)$ for every (weak) oddomorphism $\varphi \colon F \to G$ \cite{seppelt_homomorphism_2024}.
However, the only known proof again exploits a characterization of homomorphism indistinguishability over the graphs of pathwidth~$\leq k$ from finite model theory \cite{montacute_pebble-relation_2024}.
In the following, we give an alternative proof that avoids any tools from finite model theory, and instead rests on the characterization of pathwidth in terms of so-called searches.
As in \cref{thm:oddo-tw}, we prove the following slightly stronger statement.

\begin{theorem}
    \label{thm:oddo-pw}
    Let $\varphi\colon F \to G$ be an oddomorphism, and let $(P,\beta)$ be a path decomposition of $F$.
    Then
    \[\pw(G) \leq \max_{v \in V(P)} |\varphi(\beta(v))| - 1.\]
\end{theorem}

Here, a \emph{path decomposition} of a graph $F$ is a tree decomposition $(T,\beta)$ of $F$ in which $T$ is a path.
Havens are of no use for bounding the pathwidth of $G$, since a graph of small treewidth may have arbitrarily large pathwidth.
Instead, the appropriate notion is the following search game from \cite{bienstock_quickly_1991}.
A team of searchers occupies the set $X_i$ of vertices of $G$ at stage~$i$, while~$B_i$ records the places where the person they are looking for may currently be.

\begin{definition}[{\cite[Section~5]{bienstock_quickly_1991}}]\label{def:search}
    Let $G$ be a graph.
    A \emph{search} in $G$ is a sequence $(X_1,\dots,X_m)$ of subsets of $V(G)$ such that $X_1 = \emptyset$ and, for every $1 \leq i < m$, either $X_{i+1} \subseteq X_i$ or $X_i \subseteq X_{i+1}$.
    Let $B_1 \coloneqq V(G)$ and, inductively, let $B_i$ denote the set of all vertices $v$ of $G$ such that there is a path $P$ of $G$ between $v$ and some vertex of $B_{i-1}$ with $V(P) \cap X_i = \emptyset$.
    The search is \emph{successful} if $B_m = \emptyset$.
\end{definition}

Equivalently, $B_i$ is the union of all connected components of $G - X_i$ which meet $B_{i-1} \setminus X_i$.
Note that the sets $B_i$ may grow again; a search is not required to be monotone.

\begin{theorem}[{\cite[(5.1)]{bienstock_quickly_1991}}]\label{thm:search-pw}
    Let $G$ be a graph and $h \geq 0$.
    Then $\pw(G) \leq h-1$ if, and only if, there is a successful search $(X_1,\dots,X_m)$ in $G$ with $|X_i| \leq h$ for all $i \in [m]$.
\end{theorem}

This concludes the preparations for the proof of \cref{thm:oddo-pw}.
\begin{proof}[Proof of \Cref{thm:oddo-pw}]
We write the path decomposition $(P, \beta)$ of $F$ as $(p_1,\dots,p_n,\beta)$, where $p_1,\dots,p_n$ is the natural order on $V(P)$.
Call a path decomposition \emph{nested} if $\beta(p_1) = \emptyset$ and, for every $2 \leq i \leq n$, either $\beta(p_{i-1}) \subseteq \beta(p_i)$ or $\beta(p_i) \subseteq \beta(p_{i-1})$.
Every path decomposition can be turned into a nested one by prepending an empty bag and inserting the bag $\beta(p_{i-1}) \cap \beta(p_i)$ between any two consecutive bags.
All bags of the resulting path decomposition are subsets of bags of the original one, so this operation does not increase $\max_{i} |\varphi(\beta(p_i))|$, and it suffices to prove \Cref{thm:oddo-pw} for nested path decompositions.

Fix an oddomorphism $\varphi \colon F \to G$ and a nested path decomposition $(p_1,\dots,p_n,\beta)$ of~$F$, and define, for $i \in [n]$,
\begin{align*}
 L_i &\coloneqq \beta(p_1) \cup \dots \cup \beta(p_i),\\
X_i &\coloneqq \varphi(\beta(p_i)),\\
\Lambda_i &\coloneqq \left\{ u \in V(G) \;\big|\; \varphi^{-1}(u) \cap L_i \text{ contains an odd number of } \varphi\text{-odd vertices} \right\}.
\end{align*}

The set $\Lambda_i$ is the part of $G$ which the searchers will have cleared after $i$ stages, and $X_i$ is where they will stand.
The two following claims say that this is consistent: 
$X_i$ shields $\Lambda_i$ from the rest of $G$, and $\Lambda_i$ changes only inside $X_i$.

\begin{claim}
    \label{lem:pw-separation}
    For every $i \in [n]$, there is no edge of $G$ from $\Lambda_i \setminus X_i$ to $V(G) \setminus (\Lambda_i \cup X_i)$.
\end{claim}

\begin{claimproof}
    For $i = n$ we have $L_n = V(F)$, so $\Lambda_n = V(G)$ because every fibre of $\varphi$ contains an odd number of $\varphi$-odd vertices, and there is nothing to show.
    So let $i < n$ and let $A_1 \coloneqq L_i$ and $A_2 \coloneqq \beta(p_{i+1}) \cup \dots \cup \beta(p_n)$.
    It is well-known that $(A_1,A_2)$ is a separation in $F$ (see, e.g., \cite[Lemma~12.3.1]{diestel_graph_2025}) and $S \coloneqq A_1 \cap A_2 = \beta(p_i) \cap \beta(p_{i+1})$.

    Now, let $C$ be a connected component of $G - \varphi(S)$ and let $v \in C$.
    Then $v \notin \varphi(S)$, so $\varphi^{-1}(v) \cap S = \emptyset$, and hence $\varphi^{-1}(v)$ is the disjoint union of $\varphi^{-1}(v) \cap A_1$ and $\varphi^{-1}(v) \cap A_2$.
    Since $\varphi^{-1}(v)$ contains an odd number of $\varphi$-odd vertices, exactly one of these two sets does.
    By \Cref{lem:oddo-separation}, there is a $j \in \{1,2\}$ such that this is the set $\varphi^{-1}(v) \cap A_j$, simultaneously for all $v \in C$.
    Since $A_1 = L_i$, this means that $C \subseteq \Lambda_i$ if $j = 1$, and $C \cap \Lambda_i = \emptyset$ if $j = 2$.
    Hence, $\Lambda_i \setminus \varphi(S)$ is a union of connected components of $G - \varphi(S)$, and there is no edge of $G$ from $\Lambda_i \setminus \varphi(S)$ to $V(G) \setminus (\Lambda_i \cup \varphi(S))$.
    Since $S \subseteq \beta(p_i)$, and hence $\varphi(S) \subseteq X_i$, the claim follows.
\end{claimproof}

\begin{claim}
    \label{lem:pw-steps}
    It holds that $\Lambda_1 = \emptyset$, $\Lambda_n = V(G)$, and $\Lambda_{i-1} \symdiff \Lambda_i \subseteq X_i$ for all $2 \leq i \leq n$.
    Moreover, $(X_1,\dots,X_n)$ is a search in $G$.
\end{claim}

\begin{claimproof}
    Since the path decomposition is nested, $L_1 = \beta(p_1) = \emptyset$, and hence $\Lambda_1 = \emptyset$ and $X_1 = \emptyset$.
    As noted above, $L_n = V(F)$ gives $\Lambda_n = V(G)$.

    Let $2 \leq i \leq n$ and let $u \in \Lambda_{i-1} \symdiff \Lambda_i$.
    Since $L_{i-1} \subseteq L_i$, the sets $\varphi^{-1}(u) \cap L_{i-1}$ and $\varphi^{-1}(u) \cap L_i$ differ precisely in $\varphi^{-1}(u) \cap (L_i \setminus L_{i-1})$, which therefore contains an odd number of $\varphi$-odd vertices and is in particular non-empty.
    As $L_i \setminus L_{i-1} \subseteq \beta(p_i)$, this gives $u \in \varphi(\beta(p_i)) = X_i$.

    Finally, being nested, the decomposition satisfies $\beta(p_{i-1}) \subseteq \beta(p_i)$ or $\beta(p_i) \subseteq \beta(p_{i-1})$, and hence $X_{i-1} \subseteq X_i$ or $X_i \subseteq X_{i-1}$, for every $2 \leq i \leq n$.
    Together with $X_1 = \emptyset$, this shows that $(X_1,\dots,X_n)$ is a search in $G$.
\end{claimproof}

    Let $h \coloneqq \max_{i \in [n]} |\varphi(\beta(p_i))|$.
    By \Cref{lem:pw-steps}, the sequence $(X_1,\dots,X_n)$ is a search in $G$, and $|X_i| \leq h$ for all $i \in [n]$.
    We show that this search is successful, which implies that $\pw(G) \leq h-1$ by \Cref{thm:search-pw}.
    Let $B_1,\dots,B_n$ be as in \Cref{def:search}.
    We show by induction on $i$ that
    \begin{equation}\label{eq:pw-invariant}
        B_i \subseteq V(G) \setminus (\Lambda_i \cup X_i).
    \end{equation}
    For $i = 1$ this holds because $B_1 = V(G)$, while $\Lambda_1 = X_1 = \emptyset$ by \Cref{lem:pw-steps}.

    So let $2 \leq i \leq n$ and suppose that \eqref{eq:pw-invariant} holds for $i-1$.
    We first argue that $B_{i-1} \setminus X_i \subseteq V(G) \setminus (\Lambda_i \cup X_i)$.
    Indeed, let $v \in B_{i-1} \setminus X_i$.
    Then $v \notin X_i$, and $v \notin \Lambda_{i-1}$ by the induction hypothesis.
    If $v \in \Lambda_i$, then $v \in \Lambda_i \setminus \Lambda_{i-1} \subseteq X_i$ by \Cref{lem:pw-steps}, a contradiction.

    By \Cref{lem:pw-separation}, every connected component of $G - X_i$ is contained in $\Lambda_i \setminus X_i$ or in $V(G) \setminus (\Lambda_i \cup X_i)$.
    By the previous paragraph, those components which meet $B_{i-1} \setminus X_i$ are of the second kind.
    Since $B_i$ is the union of exactly these components, this proves \eqref{eq:pw-invariant}.

    Now $\Lambda_n = V(G)$ by \Cref{lem:pw-steps}, so \eqref{eq:pw-invariant} gives $B_n = \emptyset$, i.e., the search is successful.
\end{proof}

	\section{Separators}
	\label{sec:separators}
	This section is concerned with 
the interaction of oddomorphisms $F \to G$ with separators in the graph $F$.
The main technical contribution is \cref{lem:separators}.
As corollaries, we show \cref{thm:forest-main} and \cref{it:cactus,it:k33,it:k5} of \cref{thm:strong-roberson}.

\begin{lemma}[Separator lemma] \label{lem:separators}
    Let $\phi \colon F \to G$ be an oddomorphism.
    For a set of vertices $S \subseteq V(F)$, write $C_1, \dots, C_n$ for the connected components of $F - S$
    and $D_1, \dots, D_m$ for the connected components of $G - \phi(S)$.

    If $n > m$,
    then there exists a minor $F'$ of the graph obtained from $F[C \cup S]$ by adding a clique on $S$ such that 
    $F'$ admits an oddomorphism to $G$.
    Here, $C \coloneqq \bigcup_{i \in I} C_i$ for some proper subset $I \subsetneq [n]$.
\end{lemma}

\Cref{lem:separators} will be applied in \cref{sec:clique-sums,sec:extensions}
for showing that a minor-closed graph class $\mathcal{F}$ formed by taking clique-sums is closed under weak oddomorphisms.
Its prototypical application is the following \cref{lem:oddo-to-kn-clique-sum}.

\begin{corollary}
    \label{lem:oddo-to-kn-clique-sum}
    Let $s \geq 0$ and let $\mathcal{F}$ be a minor-closed graph class.
    Let $G$ be an $(s+1)$-connected graph.
    If there exists a graph $F \in \mathcal{F}^{\oplus s}$ admitting an oddomorphism to $G$,
    then there exists a graph $F' \in \mathcal{F}$ admitting an oddomorphism to $G$.
\end{corollary}

\begin{proof}[Proof (assuming \cref{lem:separators})]
    Let $F \in \mathcal{F}^{\oplus s}$ have the minimal number of vertices among graphs in $\mathcal{F}^{\oplus s}$ admitting an oddomorphism to $G$.
    Write $\phi \colon F \to G$ for the oddomorphism.
    
    If $F \not\in \mathcal{F}$, then there exists a set $S \subseteq V(F)$ such that $|S| \leq s$, 
    the graph $F - S$ is disconnected, and $F = F_1 \oplus_S F_2$ for some $F_1, F_2 \in \mathcal{F}^{\oplus s}$.
    Since $G$ is $(s+1)$-connected and $|\phi(S)| \leq |S| \leq s$,
    the graph $G - \phi(S)$ is connected.
    By \cref{lem:cliquesum-minor-closedness,lem:separators},
    there exists a graph $F' \in \mathcal{F}^{\oplus s}$ on fewer vertices than $F$ admitting an oddomorphism to $G$, contradicting the minimality of $F$.
    Hence, $F \in \mathcal{F}$.
\end{proof}

As a warm-up for \cref{lem:separators}, we prove the following lemma for oddomorphism and cut vertices.
It will be used in the proof of \cref{lem:top-minor-tree}.

\begin{lemma}[Cut vertex lemma]\label{lem:cutvertex}
    Let $\phi \colon F\to G$ be an oddomorphism.
    Let $a \in V(F)$ and write $C_1, \dots, C_n$ for the connected components of $F - a$.
    If the graph $G - \phi(a)$ is connected,
    then there exists an oddomorphism $\phi_i \colon F[C_i \cup \{a\}] \to G$ for some $i \in [n]$ such that
    \begin{enumerate}
        \item $\varphi_i(b) = \varphi(b)$ for all $b \in C_i \cup \{a\}$,
        \item for all $b \in C_i$, the vertex $b$ is $\phi$-odd if, and only if, it is $\phi_i$-odd,
        \item if $a$ is $\phi_i$-odd, then there exists a $\phi$-odd vertex $a^* \in V(F) \setminus C_i$ such that $\phi(a) = \phi(a^*)$.
    \end{enumerate}
\end{lemma}
\begin{proof}
    As the first step, we select the index $i \in [n]$.
    \begin{claim}\label{cl:select-c-simple}
        There exists an index $i \in [n]$
        such that, for every $x \in V(G) \setminus \{\phi(a)\}$,
        the set 
        $\phi^{-1}(x) \cap C_i$ contains an odd number of $\phi$-odd vertices.
    \end{claim}
    \begin{claimproof}
        This essentially follows from arguments in the proof of \cite[Lemma~3.12]{roberson_oddomorphisms_2022}.
        Concretely, let $x \in V(G) \setminus \{\phi(a)\}$
        be arbitrary.
        Since $\phi^{-1}(x) = \biguplus_{i =1}^n C_i \cap \phi^{-1}(x)$ contains an odd number of $\phi$\nobreakdash-odd vertices,
        there exists an index $i \in [n]$ such that $C_i \cap \phi^{-1}(x)$ contains an odd number of $\phi$-odd vertices.
        Let $y \in V(G) \setminus \{\phi(a)\}$ be a neighbour of $x$.
        Since $C_i \cap \phi^{-1}(x)$ contains an odd number of $\phi$-odd vertices,
        the number of edges between $C_i \cap \phi^{-1}(x)$  and $C_i \cap \phi^{-1}(y)$
        is odd.
        Hence, $C_i \cap \phi^{-1}(y)$ also contains an odd number of $\phi$-odd vertices.
        Since $G - \phi(a)$ is connected, the claim follows by iterating this argument.
    \end{claimproof}
    
    Define $\phi_i$ as the restriction of $\phi$ to $F[C_i \cup \{a\}]$.
    This map is a homomorphism and satisfies the first assertion.
    For the second assertion, observe that, for every $b \in C_i$, it holds that $N_{F[C_i \cup \{a\}]}(b) = N_F(b)$.
    Hence, for all $b \in C_i$, the vertex $b$ is $\phi$-odd if, and only if, it is $\phi_i$-odd.

    In order to show that $\phi_i$ is an oddomorphism, it remains to consider the parity of the vertex~$a$.
    Let $x \in N_G(\phi(a))$.
    Since $\phi^{-1}(x) \cap C_i$ contains an odd number of $\phi$-odd vertices,
    the number of edges in $F$ between $\phi^{-1}(x) \cap C_i$ and $\phi^{-1}(\phi(a)) \cap (C_i \cup \{a\})$ is odd.
    The $\phi_i$-parity of $a$ is thus the parity of the number of edges between $\phi^{-1}(x) \cap C_i$ and $\phi^{-1}(\phi(a)) \cap \{a\}$,
    whose parity depends only on the number of $\phi$-odd vertices in $\phi^{-1}(\phi(a)) \cap C_i$
    and not on $x$.
    Hence, $a$ is either $\phi_i$-even or $\phi_i$-odd, and $\phi_i$ is an oddomorphism.

    For the last assertion,
    suppose that $a$ is $\phi_i$-odd and $\phi$-even.
    For $x \in N_G(\phi(a))$,
    observe the following:
    \begin{enumerate}
        \item The number $\alpha$ of edges in $F$ between $\phi^{-1}(x)$ and $\phi^{-1}(\phi(a))$ is odd, since $\phi$ is an oddomorphism.
        \item The number $\beta$ of edges in $F$ between $\phi^{-1}(x) \cap C_i$ and $\phi^{-1}(\phi(a)) \cap (C_i \cup \{a\})$ is odd, since $\phi^{-1}(x) \cap C_i$ contains an odd number of $\phi$-odd vertices and $N_F(C_i) \subseteq C_i \cup \{a\}$.
        \item The vertex $a$ has an even number $\gamma$ of neighbours in $\phi^{-1}(x)$
        and an odd number $\delta$ of neighbours in $\phi^{-1}(x) \cap C_i$.
        Hence, it has an odd number $\gamma - \delta$ of neighbours in $\phi^{-1}(x) \setminus C_i$.
    \end{enumerate}
    The number of edges from $\phi^{-1}(x) \setminus C_i$ to $\phi^{-1}(\phi(a)) \setminus (C_i \cup \{a\})$ is
    $\alpha - \beta - (\gamma-\delta) \equiv 1 \mod 2$.
    Hence, there is an odd number of $\phi$-odd vertices in $\phi^{-1}(\phi(a)) \setminus (C_i \cup \{a\})$, as desired.
\end{proof}

Analogous to \cref{lem:oddo-to-kn-clique-sum},
we conclude the following \cref{lem:oddo-to-kn-one-sum} from \cref{lem:cutvertex}.

\begin{corollary}
    \label{lem:oddo-to-kn-one-sum}
    Let $\mathcal{F}$ be closed under taking subgraphs.
    Let $G$ be a $2$-connected graph.
    If there exists a graph $F \in \mathcal{F}^{\oplus 1}$ admitting an oddomorphism to $G$,
    then there exists a graph $F' \in \mathcal{F}$ admitting an oddomorphism to $G$.
\end{corollary}

\begin{proof}
    Let $F \in \mathcal{F}^{\oplus 1}$ have the minimal number of vertices of any graph in $\mathcal{F}^{\oplus 1}$ admitting an oddomorphism to $G$.
    Write $\phi \colon F \to G$ for the oddomorphism.
    
    If $F \not\in \mathcal{F}$, then there exists a vertex $a \in V(F)$ such that $F - a$ is disconnected.
    Since $G$ is $2$-connected, the graph $G - \phi(a)$ is connected.
    By \cref{lem:cutvertex},
    there exists a graph $F' \in \mathcal{F}^{\oplus 1}$ on fewer vertices than $F$ admitting an oddomorphism to $G$, contradicting the minimality of $F$.
    Hence, $F \in \mathcal{F}$.
\end{proof}

This concludes the preparations for the proof of \cref{lem:separators}.

\begin{proof}[Proof of \cref{lem:separators}]
    In order to construct the minor $F'$,
    we select a proper subset of the connected components $C_1, \dots, C_n$ of $F - S$ (analogous to \cref{cl:select-c-simple}).
    Crucially, this selection must be such that every fibre $\phi^{-1}(x)$ for $x \in V(G) \setminus \phi(S)$ contains an odd number of $\phi$-odd vertices from this set. 
    To that end, write $\mathbb{F}_2$ for the $2$-element field.
    Define a matrix $P \in \mathbb{F}_2^{m \times n}$ whose entry $P_{ji}$ is equal to the parity of the number of $\phi$-odd vertices in $\phi^{-1}(x) \cap C_i$ for any $x \in D_j$.

    \begin{claim}
        The matrix $P$ is well-defined.
    \end{claim}
    \begin{claimproof}
        Fix $i \in [n]$ and $j \in [m]$.
        Let $x, y \in D_j$ be adjacent in $G$. 
        The parity of the number of edges with endpoints in $\phi^{-1}(x) \cap C_i$ and $\phi^{-1}(y) \cap C_i$
        is equal to the parity of the number of $\phi$-odd vertices in $\phi^{-1}(x) \cap C_i$
        and also to the parity of the number of $\phi$-odd vertices in $\phi^{-1}(y) \cap C_i$.
        Since $G[D_j]$ is connected,
        this quantity does not depend on the choice of $x \in D_j$ but merely on~$j$.
    \end{claimproof}

    \begin{claim}\label{cl:select-c}
        There exists a proper subset $I \subsetneq [n]$
        satisfying that $\sum_{i \in I} P_{ji} \equiv 1 \mod 2$ for all $j \in [m]$.
    \end{claim}
    \begin{claimproof}
        Throughout this proof, arithmetic is over $\mathbb{F}_2$.
        The matrix $P$ has the property that $P \boldsymbol{1} = \boldsymbol{1}$, i.e.\ $\sum_{i \in [n]} P_{ji} =1$ for all $j \in [m]$.
        This is because every $\phi$-fibre $\phi^{-1}(x)$ for $x \in V(G) \setminus \phi(S)$ contains an odd number of $\phi$-odd vertices and the $C_1 \uplus \dots \uplus C_n$ partition the fibre.

        Hence, $Pe_1 + \dots + Pe_n = \boldsymbol{1}$ where $e_i \in \mathbb{F}_2^n$ denotes the $i$-th standard basis vector.
        Since $n > m$,
        the vectors $Pe_1, \dots, Pe_n \in \mathbb{F}_2^m$ are linearly dependent.
        Thus, there exists a non-empty set $I' \subseteq [n]$ such that $\sum_{i \in I'} P e_i = 0$.
        Hence,
        \[
         \boldsymbol{1} = Pe_1 + \dots + Pe_n
         = \sum_{i \in I'} P e_i + \sum_{i \in [n] \setminus I'} P e_i
         = \sum_{i \in [n] \setminus I'} P e_i.
        \]
        Thus, the set $I \coloneqq [n] \setminus I'$ is as desired.
    \end{claimproof}

    Let $I \subsetneq [n]$ denote the set from \cref{cl:select-c}.
    Define $C \coloneqq \bigcup_{i \in I} C_i$.
    By \cref{cl:select-c}, $C \cap \phi^{-1}(x)$  contains an odd number of $\phi$-odd vertices for every $x \in V(G) \setminus \phi(S)$.

    Having picked $C$, 
    we now define the minor $F'$ of the graph obtained from $F[C \cup S]$ by adding a clique on $S$.    
    To that end, let $\mathfrak{S} \coloneqq \left\{S \cap \phi^{-1}(y) \mid y \in \phi(S) \right\}$
    denote the partition of $S$ induced by~$\phi$.
    The vertex set of $F'$ is $C \uplus \mathfrak{S}$.
    Its edges are specified as follows:
    \begin{itemize}
        \item all edges of $F[C]$ are edges of $F'$,
        \item $vT \in E(F')$ for every $v \in C$ and $T \in \mathfrak{S}$ such that the number $|E_F(v, T)|$ of edges in $F$ from $v$ to $T$ is odd,
        \item two distinct vertices $T, T' \in \mathfrak{S}$ are adjacent if $|E_F(T, T')|$ is odd. 
    \end{itemize}

    Let $\phi' \colon F' \to G$ denote the homomorphism induced by~$\phi$.
    By the definition of $\mathfrak{S}$,
    the map
    $\phi' \colon V(F') \to V(G)$ is well-defined.
    Inspecting the three cases of the definition of the edges of $F'$
    shows that $\phi'$ is a homomorphism.
    Clearly, $F'$ is a minor of the graph obtained from $F[C \cup S]$ by adding a clique on $S$.
    It remains to verify that $\phi'$ is an oddomorphism.

    \begin{claim}\label{cl:sep1}
        If $v \in C$ is $\phi$-even, respectively $\phi$-odd, then
        it is $\phi'$\nobreakdash-even, respectively $\phi'$\nobreakdash-odd.
    \end{claim}
    \begin{claimproof}Since $C$ is a union of connected components of $F - S$,
    	all neighbours of $v \in C$ in $F$ lie in $C \uplus S$. 
    	Let $x$ be a neighbour of $\phi(v) = \phi'(v)$.
    	Hence,
    	\allowdisplaybreaks
    	\begin{align*}
    		|N_F(v) \cap \phi^{-1}(x)| &= 
    		|N_F(v) \cap \phi^{-1}(x) \cap C| + 
    		|N_F(v) \cap \phi^{-1}(x) \cap S| \\
    		&= 
    		|N_{F'}(v) \cap \phi^{-1}(x) \cap C| + 
    		|E_F(v, \phi^{-1}(x) \cap S)| \\ 
    		&\equiv |N_{F'}(v) \cap \phi^{-1}(x)| \mod 2.
    	\end{align*}
    	For the last equality, observe that, if $x \not\in \phi(S)$, then
    	$\phi^{-1}(x) \cap S = \emptyset$.
    	If $x \in \phi(S)$, let $T \coloneqq \phi^{-1}(x) \cap S$.
    	By definition, $vT \in E(F')$ if, and only if, $|E_F(v,T)|$ is odd.
    \end{claimproof}

    \begin{claim}\label{cl:sep2}
        For every $v \in V(G) \setminus \phi(S)$, the fibre $\phi'^{-1}(v)$ contains an odd number of $\phi'$-odd vertices.
    \end{claim}
    \begin{claimproof}
        By \cref{cl:sep1}, the $\phi$- and $\phi'$-parities of the vertices in question are the same.
        The claim follows from \cref{cl:select-c}.
    \end{claimproof}

    \begin{claim}\label{cl:sep3}
        Let $T \in \mathfrak{S}$.
        The $\phi'$-parity of $T$ is $1 + |\{a \in C \cap \phi'^{-1}(\phi'(T)) \mid a \text{ is $\phi$-odd}\}| \bmod 2$.
    \end{claim}
    \begin{claimproof} Let $x$ be adjacent to $\phi'(T)$.
    	First suppose that $x \not\in \phi(S)$.
    	The following observations are all via double counting.
    	The number of edges between $\phi'^{-1}(x)$ and $\phi'^{-1}(\phi'(T))$
    	is odd by \cref{cl:sep2} and the fact that $\phi$ is an oddomorphism.
    	Furthermore, these edges can be partitioned into those between $\phi'^{-1}(x)$ and $T$ and those between $\phi'^{-1}(x)$ and $\phi'^{-1}(\phi'(T)) \cap C$.
    	The latter set of edges has the same parity as the number of $\phi$-odd vertices in $\phi'^{-1}(\phi'(T)) \cap C$.
    	In symbols,
    	\allowdisplaybreaks
    	\begin{align*}
    	    |N_{F'}(T) \cap \phi'^{-1}(x)| 
    		&\equiv \sum_{a \in \phi'^{-1}(x)} |E_F(a, T)|\\
    		&= |E_F(\phi'^{-1}(x) \cap C, T)| \\
    		&= |E_F(\phi'^{-1}(x) \cap C, \phi'^{-1}(\phi'(T)))| - |E_F(\phi'^{-1}(x) \cap C, \phi'^{-1}(\phi'(T)) \cap C)|\\
    		&\equiv 1 + |\{a \in C \cap \phi'^{-1}(\phi'(T)) \mid a \text{ is $\phi$-odd}\}| \mod 2.
    	\end{align*}
    	Now suppose that $x \in \phi(S)$.
    	Write $T' \coloneqq \phi^{-1}(x) \cap S$.
    	Once again, the edges in $F'$ between $\phi'^{-1}(x)$ and $\phi'^{-1}(\phi'(T))$ are double counted.
    	\begin{align*}
    		 |N_{F'}(T) \cap \phi'^{-1}(x)|&\equiv |E_F(T, T')| + \sum_{a \in \phi^{-1}(x) \cap C} |E_F(a, T)| \\
    		&\equiv 1 + |\{a \in C \cap \phi'^{-1}(\phi'(T)) \mid a \text{ is $\phi$-odd}\}| \mod 2. \qedhere
    	\end{align*}
    \end{claimproof}

    By \cref{cl:sep1,cl:sep3}, every vertex of $F'$ is  odd or even with respect to $\phi'$.
    By \cref{cl:sep2}, it remains to verify the number of odd vertices in the fibres of the vertices in $\phi(S)$.
    Let $x \in \phi(S)$ and $T \coloneqq \phi^{-1}(x) \cap S$.
    By \cref{cl:sep1,cl:sep3}, the number of $\phi'$-odd vertices in $\phi'^{-1}(x)$ is equal to the sum of $|\{a \in C \cap \phi'^{-1}(\phi'(T)) \mid a \text{ is $\phi$-odd}\}|$ and the parity of $T$, i.e.,
    \[
        1 + 2 |\{a \in C \cap \phi'^{-1}(\phi'(T)) \mid a \text{ is $\phi$-odd}\}| \equiv 1 \mod 2.
    \]
    Hence, $\phi'$ is an oddomorphism.
\end{proof}

	\subsection{Classes of Forests Closed under Topological Minors}

Applying \cref{lem:cutvertex}, we prove \cref{thm:forest-main} and thus resolve \cref{conj:strong-roberson} for all classes of forests.
Let $F$ and $G$ be graphs.
We say that $\rho\colon V(G) \to V(F)$ is a \emph{topological model} of~$G$ in~$F$ if there is a collection of internally vertex-disjoint paths $(P_{vw})_{vw \in E(G)}$ in $F$ from $\rho(v)$ to $\rho(w)$.
If there exists a topological model of $G$ in $F$,
then $G$ is said to be a \emph{topological minor} of $F$.

\begin{theorem}[\cref{thm:forest-main}]
    Every class $\mathcal{F}$ of forests closed under topological minors is closed under weak oddomorphisms. 
    If $\mathcal{F}$ is closed under disjoint unions,
    then $\mathcal{F}$ is homomorphism distinguishing closed.
\end{theorem}

The theorem follows from \cref{lem:oddo-connected}, \cref{thm:rob62}, and the following \cref{lem:top-minor-tree}.

\begin{lemma}\label{lem:top-minor-tree}
    Let $\varphi\colon F \to G$ be an oddomorphism.
    If $F$ is a tree,
    then $G$ is a topological minor of~$F$.
\end{lemma}

\begin{proof}
    By \cref{lem:oddo-planar-degree}, the graph $G$ is a tree.
    We prove by induction on $|V(F)|$ that there is a topological model $\rho\colon V(G) \to V(F)$ of $G$ in $F$ such that
    \begin{enumerate}
        \item $\varphi(\rho(v)) = v$ for all $v \in V(G)$, and\label{it:topmodel1}
        \item $\rho(v)$ is $\varphi$-odd for all $v \in V(G)$.\label{it:topmodel2}
    \end{enumerate}
    In the base case $|V(F)| = 1$, it holds that $F \cong K_1 \cong G$ since $\phi$ is surjective.
    So suppose that $|V(G)| \geq 2$ and let $v \in V(G)$ be a leaf of $G$.
    Let $X \coloneqq\phi^{-1}(v)$ be the fibre of $v$.
    
    We distinguish two cases.
    First suppose that every $a \in X$ is a leaf of $F$, i.e., $\deg_F(a) = 1$ for every $a \in X$.
    Then $F - X$ is a tree.
    By \cref{lem:oddo-connected}, the map $\varphi|_{F - X}$ is an oddomorphism from $F - X$ to $G - v$.
    By the inductive hypothesis, there is a topological model $\rho\colon V(G)\setminus \{v\} \to V(F)\setminus X$ of $G - v$ in $F - X$ satisfying \cref{it:topmodel1,it:topmodel2}.
    Let $w$ be the unique neighbour of $v$ in $G$, and let $b \coloneqq \rho(w)$.
    Note that $b$ is $\varphi$-odd, so it has a neighbour $a \in X$.
    Since $X$ only contains vertices of degree $1$, we conclude that $a$ is $\varphi$-odd.
    We set $\rho(v) \coloneqq a$ and obtain the desired topological model of $G$ in $F$.
    
    Next, suppose that there is some $a \in X$ such that $\deg_F(a) \geq 2$.
    In particular, $F - a$ is disconnected, since $F$ is a tree.
    Let $C_1,\dots,C_n$ denote the vertex sets of the connected components of $F - a$.
    By \cref{lem:cutvertex},
   	there is some index $i \in [n]$ and an oddomorphism $\varphi_i\colon F[C_i \cup \{a\}] \to G$ such that
    \begin{enumerate}
        \item $\varphi_i(b) = \varphi(b)$ for all $b \in C_i \cup \{a\}$,
        \item $b$ is $\varphi_i$-odd if and only if $b$ is $\varphi$-odd for all $b \in C_i$, and
        \item if $a$ is $\varphi_i$-odd, but $a$ is $\varphi$-even, then there is some $j \in [n] \setminus \{i\}$ such that $C_j \cap X$ contains a $\varphi$-odd vertex $a^*$.
    \end{enumerate}
    We apply the inductive hypothesis to $\varphi_i$ and obtain a topological model $\rho_i \colon V(G) \to C_i \cup \{a\}$.
    We define the map $\rho$ to be identical to $\rho_i$
    except when $a$ is $\varphi_i$-odd and $a$ is $\varphi$-even.
    In this case, let $\rho(v) \coloneqq a^*$.
    Since $a^* \in C_j \neq C_i$,
    there is a path from $a^*$ to $a$ whose internal vertices lie in $C_j$.
    This path is internally vertex-disjoint from the other paths in~$\rho_i$.
\end{proof}

	\subsection{Clique-Sums}
	\label{sec:clique-sums}
	This section is concerned with proving, under certain assumptions on $\mathcal{F}$,
that if $\mathcal{F}$ is closed under weak oddomorphisms,
then so is the class $\mathcal{F}^{\oplus s}$ of graphs that are $s$-sums of graphs in~$\mathcal{F}$.
For $s \leq 2$, we show that the minimal forbidden subgraphs or minors of $\mathcal{F}^{\oplus s}$ must be $(s+1)$-connected.
Then, \cref{lem:oddo-to-kn-clique-sum} implies that any oddomorphism from $\mathcal{F}^{\oplus s}$ to a graph $G \notin \mathcal{F}^{\oplus s}$ yields an oddomorphism from $\mathcal{F}$ to $G$.

\subsubsection[1-Sums]{$1$-Sums}

In the case of $1$-sums,
it suffices to assume that $\mathcal{F}$ is closed under taking subgraphs.
A \emph{minimal excluded subgraph} of a graph class $\mathcal{F}$ is a graph $G \not\in \mathcal{F}$ such that every proper subgraph $G'$ of $G$ is in $\mathcal{F}$.

\begin{lemma}\label{lem:min-excluded-1sum}
	If a graph class $\mathcal{F}$ is closed under taking subgraphs,
	then every minimal excluded subgraph of $\mathcal{F}^{\oplus 1}$ is $2$-connected.
\end{lemma}
\begin{proof}
	Let $G$ be a minimal excluded subgraph of $\mathcal{F}^{\oplus 1}$.
	If $G$ is disconnected,
	then, by minimality, its connected components are in $\mathcal{F}^{\oplus 1}$.
	Since $\mathcal{F}^{\oplus 1}$ is closed under disjoint unions,
	it follows that $G \in \mathcal{F}^{\oplus 1}$, a contradiction.
	
	If $G$ is connected and not $2$-connected, then there exists a cut vertex $x \in V(G)$.
	In particular, $G$ is a $1$-sum as $G = G_1 \oplus_{\{x\}} \dots \oplus_{\{x\}} G_m$
	for some subgraphs $G_1, \dots, G_m$ of $G$.
	By minimality, each of $G_1, \dots, G_m$ is in $\mathcal{F}^{\oplus 1}$.
	Since $\mathcal{F}^{\oplus 1}$ is closed under $1$-sums,
	it follows that $G \in \mathcal{F}^{\oplus 1}$, a contradiction.
\end{proof}

\begin{lemma}\label{cor:one-cliquesum}
    Let $\mathcal{F}$ be closed under taking subgraphs.
    If $\mathcal{F}$ is closed under weak oddomorphisms,
    then $\mathcal{F}^{\oplus 1}$ is closed under weak oddomorphisms.
\end{lemma}
\begin{proof}Let $F \to G$ be a weak oddomorphism for $F \in \mathcal{F}^{\oplus 1}$.
	If $G \not\in \mathcal{F}^{\oplus 1}$,
	then $G$ contains a $2$-connected subgraph $G' \not\in \mathcal{F}^{\oplus 1}$ by \cref{lem:min-excluded-1sum}.
	By \cref{lem:oddo-connected},
	there exists a subgraph $F'\in \mathcal{F}^{\oplus 1}$ of $F$ and an oddomorphism $F' \to G'$.
    By \cref{lem:oddo-to-kn-one-sum},
    there exists a graph $F'' \in \mathcal{F}$ admitting an oddomorphism to $G'$.
	By assumption, $G' \in \mathcal{F}$, a contradiction.
\end{proof}

\Cref{cor:one-cliquesum} readily gives an alternative proof of \cite[Corollary~6.8]{roberson_oddomorphisms_2022} asserting that the class of all forests is homomorphism distinguishing closed, see \cref{lem:oddo-planar-degree}.
Indeed, forests are precisely the $0$- and $1$-sums of copies of $K_2$.
By \cref{obs:surjective}, the subgraph-closure $\{K_2, K_1 + K_1, K_1\}$ of $\{K_2\}$ is closed under weak oddomorphisms.
Hence, \cref{cor:one-cliquesum} applies.
This idea is generalised by the following theorem, which subsumes the case $k = 1$ of forests.

\begin{theorem}\label{thm:adhesion-one}\label{thm:tw1}
    For $k \geq 1$, 
    the class of all graphs admitting a tree decomposition of width $\leq k$ and adhesion $\leq 1$
    is homomorphism distinguishing closed.
\end{theorem}
\begin{proof} Let $\mathcal{F}$ be the componental graph class generated by all graphs on at most $k+1$ vertices.
	This graph class is closed under weak oddomorphisms.
	Indeed, let $\phi \colon F \to G$ be a weak oddomorphism with $F  \in \mathcal{F}$.
	By \cref{lem:oddo-connected}, $G$ and $F$ can be assumed to be connected.
	Hence, $k + 1 \geq |V(F)| \geq |V(G)|$ since $\phi$ is surjective.
	The graph class in question is $\mathcal{F}^{\oplus 1}$,
	which is closed under weak oddomorphisms by \cref{cor:one-cliquesum}.
	Finally, apply \cref{thm:rob62} to deduce that $\mathcal{F}^{\oplus 1}$ is homomorphism distinguishing closed.
\end{proof}

As a second example, we obtain the following \cref{thm:cactus} by applying \cref{lem:oddo-planar-degree,cor:one-cliquesum,thm:rob62}.
Here, a \emph{cactus graph} is a graph in which any two simple cycles have at most one vertex in common.
Cactus graphs are outerplanar and characterised by the forbidden minor $K_4 - e$.

\begin{theorem}\label{thm:cactus}
    For $d \geq 1$, 
    the class
    \(
        \{F \mid \Delta(F) \leq d\}^{\oplus 1}
    \)
    of $1$\nobreakdash-sums of graphs of maximum degree $\leq d$ is closed under weak oddomorphisms and homomorphism distinguishing closed.
    
    In particular, the class of cactus graphs ($d = 2$) is closed under weak oddomorphisms and homomorphism distinguishing closed.
\end{theorem}

\subsubsection[2-Sums]{$2$-Sums}

The arguments from the previous section can be lifted to $2$-sums,
here, however under the assumption that $\mathcal{F}$ is minor-closed.
As we shall see in \cref{lem:contractors-for-minors}, this restriction is necessary.

First, we strengthen \cref{lem:min-excluded-1sum}.
A \emph{minimal excluded minor} of a graph class $\mathcal{F}$ 
is a graph $G \not\in \mathcal{F}$ such that every proper minor $G'$ of $G$ is in $\mathcal{F}$.
Note that it does not hold in general that, if $\mathcal{F}$ is minor-closed, then the minimal excluded minors of $\mathcal{F}^{\oplus s}$ are $(s+1)$\nobreakdash-connected.
For example, the class of forests is closed under clique-sums of any arity. However, its minimal excluded minor $K_3$ is not $3$\nobreakdash-connected.

\begin{lemma}\label{lem:min-excluded-2sum}
	If a graph class $\mathcal{F}$ is minor-closed,
	then every minimal excluded minor of $\mathcal{F}^{\oplus 2}$ is $3$-connected.
\end{lemma}
\begin{proof}
	Let $G$ be a minimal excluded minor of $\mathcal{F}^{\oplus 2}$.
	By \cref{lem:min-excluded-1sum}, $G$ is $2$-connected.
	If it is not $3$-connected,
	then there exist two vertices $x_1 \neq x_2$ of $G$
	such that $G - \{x_1, x_2\}$ is disconnected.
	Write $D_1, \dots, D_m$ for the connected components of $G - \{x_1, x_2\}$ where $m \geq 2$.
    Since $G - x_1$ and $G - x_2$ are connected,
    it holds that $x_1, x_2 \in N_G(D_j)$ for all $j \in [m]$.
    For all $j \in [m]$,
    write $G_j$ for the minor of $G$ obtained by contracting all $D_\ell$ for $\ell \neq j$ to the edge $x_1x_2$.
    Then $G$ is the $2$-sum of $G_1, \dots, G_m$ at the clique $\{x_1, x_2\}$.
    It follows that $G_1, \dots, G_m \in \mathcal{F}^{\oplus 2}$ and hence $G \in \mathcal{F}^{\oplus 2}$, a contradiction.
\end{proof}

\begin{lemma}\label{cor:two-cliquesum}
    Let $\mathcal{F}$ be minor-closed.
    If $\mathcal{F}$ is closed under weak oddomorphisms,
    then $\mathcal{F}^{\oplus 2}$ is closed under weak oddomorphisms.
\end{lemma}

\begin{proof}
		Let $F \to G$ be a weak oddomorphism for $F \in \mathcal{F}^{\oplus 2}$.
		If $G \not\in \mathcal{F}^{\oplus 2}$,
		then, by \cref{lem:min-excluded-2sum}, there exists a $3$-connected minor $G' \not\in \mathcal{F}^{\oplus 2}$ of $G$.
		By \cref{lem:oddo-connected},
		there exists an oddomorphism $F' \to G'$ from some minor $F' \in \mathcal{F}^{\oplus 2}$ of $F$.
		By \cref{lem:oddo-to-kn-clique-sum},
		there exists a graph $F'' \in \mathcal{F}$ admitting an oddomorphism to  $G'$.
		Hence, $G' \in \mathcal{F}$, contradicting that $G' \not\in \mathcal{F}^{\oplus 2}$.
\end{proof}

As in the case of $1$-sums,
it follows readily that the class of graphs of treewidth $\leq 2$ is closed under weak oddomorphisms
as first shown in \cite[Corollary~6.9]{roberson_oddomorphisms_2022}, see \cref{lem:oddo-planar-degree}.
More generally, \cref{thm:tw2} follows via \cref{cor:two-cliquesum} as in \cref{thm:adhesion-one}.

\begin{theorem}\label{thm:tw2}
    For $k \geq 1$, 
    the class of all graphs admitting a tree decomposition of width $\leq k$ and adhesion $\leq 2$
    is homomorphism distinguishing closed.
\end{theorem}

Employing a result of Wagner~\cite{wagner_erweiterung_1937}, we furthermore prove the following, that is, \cref{it:k33} of \cref{thm:strong-roberson}.

\begin{theorem}
    The class of $K_{3,3}$-minor-free graphs is closed under weak oddomorphisms and homomorphism distinguishing closed.
\end{theorem}
\begin{proof}
    Wagner~\cite{wagner_erweiterung_1937} proved that the $K_{3,3}$-minor-free graphs are precisely those that are $2$-sums of planar graphs and $K_5$, see \cite[Table~1]{kriz_clique-sums_1990}.
    Let $\mathcal{F}$ denote the union of the class of planar graphs with $\{K_5\}$.
    In order to apply \cref{cor:two-cliquesum},
    it must be argued that $\mathcal{F}$ is closed under weak oddomorphisms and taking minors.
    Let $F \in \mathcal{F}$ and $ \phi \colon F \to G$ be a weak oddomorphism.
    If $F$ is planar, then so is $G$ by \cref{lem:oddo-planar-degree}.
    If $F$ is non-planar, then $F \cong K_5$.
    Since $\phi$ is a surjective homomorphism, $G \cong K_5 \in \mathcal{F}$.
    Hence, the claim follows from \cref{cor:two-cliquesum,thm:rob62}.
\end{proof}

\subsubsection[3-Sums]{$3$-Sums}\label{sec:k5}
We finally give an application of \cref{lem:separators} for $3$\nobreakdash-sums by showing that the class of $K_5$-minor-free graphs is homomorphism distinguishing closed.
Wagner~\cite{wagner_uber_1937} showed that these graphs are precisely the $3$-sums of planar graphs $\mathcal{P}$ and the Wagner graph $V_8$, see \cref{fig:wagner-graph}, \cite{kriz_clique-sums_1990}.

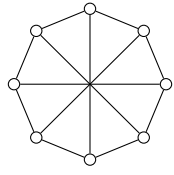
\begin{figure}
	\centering
	\begin{tikzpicture}[scale=1.2, every node/.style={circle,draw,fill=white,line width=1pt, inner sep=2.5}]
\foreach \i in {0,...,7}{
			\node (v\i) at ({cos(360/8*\i)},{sin(360/8*\i)}) {};
		}
\foreach \i in {0,...,7}{
			\pgfmathtruncatemacro{\j}{mod(\i+1,8)}
			\draw[thick] (v\i) -- (v\j);
		}
\draw[thick] (v0) -- (v4);
		\draw[thick] (v1) -- (v5);
		\draw[thick] (v2) -- (v6);
		\draw[thick] (v3) -- (v7);
	\end{tikzpicture}
	
	\caption{The Wagner graph $V_8$.}
	\label{fig:wagner-graph}
\end{figure}

\begin{lemma}[restate=lemWagner,label=lem:kfive3]
    No minor of $V_8$ admits an oddomorphism to $K_5$.
\end{lemma}

\begin{proof}
	Let $F$ be a minor of $V_8$
	and $\phi \colon F \to K_5$ be an oddomorphism.
	Each $\phi$-odd vertex has degree at least $4$.
	Since every fibre of $\phi$ contains a $\phi$-odd vertex,
	the graph $F$ must contain at least $5$ vertices of degree at least $4$.
	As $V_8$ is $3$-regular, the branch set in $V_8$ corresponding to such a vertex must contain at least two vertices.
	Since these branch sets are pairwise disjoint, $V_8$ would have at least $10$ vertices, a contradiction.
\end{proof}

\begin{theorem}\label{thm:k5}
    The class of $K_5$-minor-free graphs is closed under weak oddomorphisms and homomorphism distinguishing closed.
\end{theorem}
\begin{proof}By \cite{wagner_uber_1937},
    the class of $K_5$-minor-free graphs is $(\mathcal{P} \cup \mathcal{V})^{\oplus 3}$
    where $\mathcal{V}$ denotes the minor closure of $V_8$.
    Let $F \in (\mathcal{P} \cup \mathcal{V})^{\oplus 3}$
    and $F \to G$ be a weak oddomorphism.
    If $K_5$ is a minor of~$G$,
    then there exists, by \cref{lem:oddo-minors},
    a graph $F' \in (\mathcal{P} \cup \mathcal{V})^{\oplus 3}$
    admitting an oddomorphism to $K_5$.
    By definition,
    the graph class $\mathcal{P} \cup \mathcal{V}$ is minor-closed.
    Since $K_5$ is $4$\nobreakdash-connected,
    there exists, by \cref{lem:oddo-to-kn-clique-sum},
    a graph $F'' \in \mathcal{P} \cup \mathcal{V}$
    admitting an oddomorphism to $K_5$.
    By \cref{lem:oddo-planar-degree},
    the graph
    $F''$ is non-planar and thus a minor of~$V_8$,
    contradicting \cref{lem:kfive3}.
    Hence, $G$ is $K_5$-minor-free.
    The final assertion follows from \cref{thm:rob62}.
\end{proof}

	\subsection{Outerplanar Graphs}
	\label{sec:k2h}
	As another application of \cref{lem:separators}, we show, for every $h \geq 3$, that the class of graphs of treewidth at most two that do not have the complete bipartite graph~$K_{2,h}$ as a minor is homomorphism distinguishing closed.
The best-known example of such a graph class is the class of outerplanar graphs, which is characterised \cite{syslo_characterisations_1979} by the forbidden minors $K_4$ and $K_{2,3}$.
The arguments in this section were in slightly different form first published in the second author's dissertation \cite[Section~6.3.4]{seppelt_homomorphism_2024}.

\begin{lemma} \label{thm:k2h-main}
	Let $h \geq 3$.
	If $F$ is a graph of treewidth at most two admitting an oddomorphism $\phi \colon F \to K_{2,h}$, 
	then $K_{2,h}$ is a minor of $F$.
\end{lemma}

\Cref{thm:k2h-main} yields the following theorem via \cref{thm:tw2,lem:oddo-minors,thm:rob62}, i.e.\ \cref{it:outerplanar} of \cref{thm:strong-roberson}.

\begin{theorem} \label{cor:outerplanar-closed}
	Let $h \geq 3$.
	The class of $\{K_4, K_{2,h} \}$-minor-free graphs is closed under weak oddomorphisms and homomorphism distinguishing closed.
	In particular, the class of outerplanar graphs is closed under weak oddomorphisms and homomorphism distinguishing closed.
\end{theorem}

\subsubsection{Vertices of Low Degree Which Do Not Cut}

Towards the proof of \cref{thm:k2h-main}, several lemmas  are devised which provide vertices of low degree in graphs of bounded treewidth  that are not cut vertices.
For a tree decomposition $(T, \beta)$ of some graph $F$, define $\lVert (T, \beta) \rVert \coloneqq \sum_{t \in V(T)} |\beta(t)|^2$.
We will consider tree decompositions which are minimal with respect to this quantity.

\begin{lemma} \label{lem:tw-weight}
	Let $F = F_1 + F_2$ be a   graph with tree decomposition $(T, \beta)$.
	Define $\beta_i(t) \coloneqq \beta(t) \cap V(F_i)$ for $i \in [2]$.
	Then $(T, \beta_i)$ is a tree decomposition of $F_i$ for $i \in [2]$
	and $\lVert (T, \beta_1) \rVert + \lVert (T, \beta_2) \rVert \leq \lVert (T, \beta) \rVert$.
	Furthermore, if there exists $r \in V(T)$ such that both $\beta(r) \cap V(F_1)$ and $\beta(r) \cap V(F_2)$ are non-empty, then this inequality is strict.
\end{lemma}
\begin{proof}
	It is easily verified that $(T, \beta_i)$ is a tree decomposition of $F_i$. For the second claim,
    \allowdisplaybreaks
	\begin{align*}
		 \lVert (T, \beta_1) \rVert + \lVert (T, \beta_2) \rVert
		&= \sum_{t \in V(T)} |\beta_1(t)|^2 + \sum_{t \in V(T)} |\beta_2(t)|^2 \\
		&= \sum_{t \in V(T)} \left( |\beta(t) \cap V(F_1)|^2 + |\beta(t) \cap V(F_2)|^2 \right) \\
		&\leq \sum_{t \in V(T)} |\beta(t)|^2.
	\end{align*}
	If $\beta(r) \cap V(F_1)$ and $\beta(r) \cap V(F_2)$ are non-empty, then the inequality is strict.
\end{proof}

A \emph{cut vertex} $v \in V(F)$ is a vertex such that $F-v$ has more connected components than $F$.
The following \cref{fact:cutvertex} is well-known.

\begin{fact} \label{fact:cutvertex}
	Every  graph $F$ contains a vertex $v \in V(F)$ which is not a cut vertex.
\end{fact}
\begin{proof}
Suppose without loss of generality that $F$ is connected.
	Let $v, w \in V(F)$ be vertices such that their distance $\dist_F(v, w)$ is maximal.
	As usual, the distance of two vertices is the length of a shortest path between them.
	Let $a \in V(F) \setminus \{v,w\}$ be arbitrary.
	Then a shortest path in $F$ from $a$ to $w$ avoids $v$, since otherwise $\dist_F(a,w) = \dist_F(a,v) + \dist_F(v,w) > \dist_F(v,w)$, contradicting the maximality of $\dist_F(v,w)$.
	Hence, $F - v$ is connected.
\end{proof}

The following \cref{lem:tw-degree} yields vertices of low degree which do not cut.

\begin{lemma} \label{lem:tw-degree}
	Let $k \geq 0$.
	Let $F$ be a  graph of treewidth at most~$k$. 
	Then there exists a vertex $v \in V(F)$ of degree at most~$k$ which is not a cut vertex.
\end{lemma}
\begin{proof}
	Let $(T, \beta)$ be a tree decomposition of $F$ of width at most~$k$ such that $\lVert (T, \beta) \rVert$ is minimal.
	If $T$ has only one vertex, then $F$ has at most $k+1$ vertices and \cref{fact:cutvertex} yields the desired vertex.
	Otherwise, let $\ell \in V(T)$ be a leaf of $T$ and let $r \in V(T)$ denote the adjacent vertex.
	
	By minimality, $\beta(\ell) \not\subseteq \beta(r)$ and hence there exists a vertex $v \in \beta(\ell) \setminus \beta(r)$.
	This vertex has $\deg_F(v) \leq k$.
	Suppose that $v$ is a cut vertex of $F$ and write $F -v = F_1 + F_2$. Distinguish two cases:
	
	If $V(F_1) \cap \beta(r) = \emptyset$, then $F_1$ is a subgraph of $F[\beta(\ell)]$.
	Without loss of generality, we may suppose that $F_1$ is connected.
	If $|N_F(v) \cap V(F_1)| \geq 2$, there exists, by \cref{fact:cutvertex}, a vertex $w \in V(F_1)$  which is not a cut vertex of $F_1$.
	This vertex is not a cut vertex of $F$ since there are two neighbours of $v$ in $F_1$.
	If $|N_F(v) \cap V(F_1)| = 1$, then apply the same argument to the vertex $w \in N_F(v) \cap V(F_1)$ and $F'_1 \coloneqq F_1 - w$ instead of $F_1$.
	
	If $V(F_1) \cap \beta(r) \neq \emptyset$ and $V(F_2) \cap \beta(r) \neq \emptyset$,
	consider the tree decomposition $(T', \beta')$ for the subgraph $F - (\beta(\ell) \setminus \beta(r))$ of $F-v$  obtained by deleting the leaf~$\ell$ from $(T, \beta)$.
	Write $(T', \beta'_1)$ and $(T', \beta'_2)$ for the tree decompositions of $F_1 - \left((\beta(\ell) \setminus \beta(r)) \cap V(F_1)\right)$ and $F_2 - \left((\beta(\ell) \setminus \beta(r)) \cap V(F_2)\right)$ defined for $(T', \beta')$ as in \cref{lem:tw-weight}.
	Since $\beta'(r)$ contains elements from both $F_1$ and $F_2$, it holds that $\lVert (T', \beta'_1) \rVert + \lVert (T', \beta'_2) \rVert < \lVert (T', \beta') \rVert$.
	Let $S$ denote the tree obtained by taking two copies $T_1$ and $T_2$ of $T'$ and connecting both copies of $r$ to a new vertex $s$.
	Define a new tree decomposition $(S, \gamma)$ for $F$ by letting
	\[
	\gamma(x) \coloneqq \begin{cases}
		\beta(\ell), & \text{if } x = s,\\
		\beta'_1(x), & \text{if } x \in V(T_1), \\
		\beta'_2(x), & \text{if } x \in V(T_2).
	\end{cases}
	\]
	By \cref{lem:tw-weight},
	\begin{align*}
		\lVert (S, \gamma) \rVert = |\beta(\ell)|^2 + \lVert (T', \beta'_1) \rVert + \lVert (T', \beta'_2) \rVert
		< |\beta(\ell)|^2 + \lVert (T', \beta') \rVert 
        = \lVert (T, \beta) \rVert.
	\end{align*}
	This contradicts the minimality of $(T, \beta)$.
\end{proof}

\subsubsection{Oddomorphisms to Complete Bipartite Graphs}

Before we prove \cref{thm:k2h-main}, we establish the following lemma, which allows us to remove twins from a graph $F$ admitting an oddomorphism $\phi$ to a graph $G$ in the case when they lie in the same $\phi$-fibre.

\begin{lemma} \label{lem:remove-twins}
	Let $F$ and $G$ be  graphs admitting an oddomorphism $\phi \colon F \to G$.
	If there exist vertices $v \neq w \in V(F)$ such that $\phi(v) = \phi(w)$ and $N_F(v) = N_F(w)$, then $\phi|_{F - \{v,w\}}$ is an oddomorphism from $F - \{v,w\}$ to $G$.
\end{lemma}

\begin{proof}
	Clearly, $\psi \coloneqq \phi|_{F - \{v,w\}}$ is a homomorphism.
	Furthermore, for every vertex $a \in V(F) \setminus \{v,w\}$ and $x \in N_G(\psi(a))$,
	\[
	N_{F - \{v, w\}}(a) \cap \psi^{-1}(x)
	= 
	(N_F(a) \cap \phi^{-1}(x)) \setminus \{v, w\}.
	\]
	The parities of both sets are the same since $v \in N_F(a)$ if, and only if, $w \in N_F(a)$.
	Hence, if a vertex in $V(F) \setminus \{v,w\}$ is $\phi$-odd ($\phi$-even), then it is $\psi$-odd ($\psi$-even).
	Furthermore, $v$ and $w$ have the same parity with respect to $\phi$. Hence, $\psi^{-1}(\phi(v))$ contains an odd number of odd vertices.
	All other fibres are unaffected.
\end{proof}

This concludes the preparations for the proof of \cref{thm:k2h-main}.

\begin{proof}[Proof of \cref{thm:k2h-main}]
	Towards a contradiction, let $F$ be a minor-minimal $K_{2,h}$-minor-free graph of treewidth at most two admitting an oddomorphism $\phi \colon F \to K_{2, h}$. 
	Denote the vertices of $K_{2,h}$ by $V(K_{2,h}) = \{x_1, \dots, x_h, y_1, y_2\}$ and 
	write $X_i \coloneqq \phi^{-1}(x_i)$ for $i \in [h]$ and $Y_j \coloneqq \phi^{-1}(y_j)$ for $j \in \{1,2\}$ for the corresponding fibres along~$\phi$.
	Let $X \coloneqq \bigcup_{i \in [h]} X_i$ and $Y \coloneqq Y_1 \cup Y_2$.
	
	Write $V_2 \subseteq V(F)$ for the set of all vertices of degree two.
    \begin{claim}\label{cl:k2h-main1}
        Every vertex in $F$ has degree at least two.
        All vertices in $V_2$ are $\phi$-odd.
    \end{claim}
    \begin{claimproof}
        By \cref{lem:oddo-connected},
        $F$ is a connected graph on at least $2+h \geq 5$ vertices.
        Hence, it does not contain isolated vertices.
		Any vertex $v$ of degree one in $F$ must be $\phi$-odd. Since $\phi(v)$ is of degree at least two in $K_{2,h}$, 
		the vertex $v$ must have at least two neighbours. 
        
        If $v \in V_2$ is $\phi$-even,
        then both its neighbours $w_1, w_2$ lie in the same fibre.
        The graph $F - \{w_1, w_2\}$ has at least two connected components.
        Write $t \coloneqq \phi(w_1) = \phi(w_2)$.
        Since $K_{2,h}$ is $2$-connected,
        \cref{lem:cutvertex}
        implies that $F$ is not minimal. 
    \end{claimproof}
    
	\cref{cl:k2h-main1} implies that $V_2 \subseteq X$ because every odd vertex in $Y$ is of degree at least $h \geq 3$. 
	In particular, no two vertices in $V_2$ are adjacent.
	Consider the graph $F'$ with $V(F') \coloneqq V(F) \setminus  V_2$ obtained from $F$ by contracting, for each $w \in V_2$, one of the edges incident to $w$. 
    Formally,
	\begin{equation} \label{eq:def-f-prime}
		E(F') \coloneqq \left\{uv \ \middle|\
		    u, v \in V(F'), uv \in E(F) \lor \left(\exists x \in V_2.\ ux, xv \in E(F) \right)
		\right\}.
	\end{equation} 
	Note that $F'$ is a minor of $F$ without isolated vertices.
	
	For every vertex $w \in V(F')$, consider the map $p_w \colon N_F(w) \to N_{F'}(w)$ sending $x \in N_F(w) \cap V_2$ to the unique vertex $y \neq w$ such that $yx \in E(F)$ and $x \in N_F(w) \setminus V_2$ to $x$.
    That is, $p_w$ assigns to a neighbour of $w$ in $F$ the neighbour in $F'$ corresponding to it under the contraction from \eqref{eq:def-f-prime}.
	By definition of~$F'$, the map $p_w$ is surjective.
	Observe that $p_w|_{N_F(w) \setminus V_2}$ is injective.
    
	By \cref{lem:tw-degree}, $F'$ contains a vertex $v^*$ with $\deg_{F'}(v^*) \in \{1,2\}$ which is not a cut vertex in~$F'$.
	In preparation for a case distinction, consider the following claims.
	
	\begin{claim}\label{cl:vstar-in-Y}
		$v^* \in Y$.
	\end{claim}
	\begin{claimproof}
		Since $p_{v^*}$ is surjective, it holds that $\deg_{F'}(v^*) \leq \deg_{F}(v^*)$. By \cref{cl:k2h-main1} and since $v^* \not\in V_2$, it holds that $\deg_F(v^*) \geq 3$. Hence, there exist $x \neq y \in N_F(v^*)$ such that $p_{v^*}(x) = p_{v^*}(y)$. Since $p_w|_{N_F(w) \setminus V_2}$ is injective, without loss of generality, $x \in V_2 \subseteq X$. This implies that $v^* \in Y$.
	\end{claimproof}
	\begin{claim} \label{cl:cutvertex}
		If $v \in V(F')$ is a cut vertex of $F$, then it is a cut vertex of $F'$.
		In particular, $v^*$ is not a cut vertex of $F$.
	\end{claim}
	\begin{claimproof}
		By contraposition, suppose that, for all $a, b\in V(F') \setminus \{v\}$ which are connected in $F'$, there exists a path connecting them in $F' - v$.
		Let $a, b \in V(F) \setminus \{v\}$ be arbitrary vertices in the same connected component of $F$.
		Set $a'$ to $a$ if $a \in V(F')$ and to any $x \in N_F(a)$ such that $x \neq v$ if $a \in V_2$.
		Define $b'$ analogously. Observe that $a'$ and $b'$ lie in the same connected component of $F$.
		Hence, there exists a path connecting them. This path can be turned into a path in $F'$ by shortcutting every vertex in $V_2$ appearing on it.
		Hence, there exists a path connecting them in $F' - v$. This path can be transformed into a path in $F - v$ by replacing shortcut edges $wy \in E(F' - v)$ with walks via any associated vertex in $x \in V_2$, i.e.\ $wx, xy \in E(F - v)$, cf.\ \cref{eq:def-f-prime}.
	\end{claimproof}
	\begin{claim} \label{cl:eulerian}
		Let $I \subseteq [h]$ be a set of even size. 
		Let $H \coloneqq F[Y_1 \cup Y_2 \cup \bigcup_{i \in I} X_i]$. Then $\deg_H(v)$ is even for every $v \in V(H)$.
	\end{claim}
	\begin{claimproof}
		For every vertex $v \in V(H)$, the number of neighbours of $\phi(v)$ among $y_1, y_2$ and $x_i$ for $i \in I$ is even.
	\end{claimproof}

	By \cref{cl:vstar-in-Y}, without loss of generality, assume that $v^* \in Y_1$.
	Recall that $\deg_{F'}(v^*) \leq 2$.
	Hence, the image of $p_{v^*}$ is of size at most two.
	Observe that $p_{v^*}(N_F(v^*) \cap V_2) \subseteq Y$ and $p_{v^*}(N_F(v^*) \setminus V_2) \subseteq X$.
	By \cref{cl:k2h-main1} and  since $v^* \not\in V_2$, the degree of $v^*$ in $F$ is at least $3$.
	Hence, there must be at least one vertex from $Y$  in the image of $p_{v^*}$.
	Distinguish cases based on the parity of $v^*$ with respect to $\phi$:
	
	\begin{enumerate}
		\item $v^*$ is $\phi$-even.

	Distinguish cases:
	\begin{enumerate}
		\item The image of $p_{v^*}$ has size one.
		
		Then there exists a fibre $X_i$ containing two neighbours $u_1, u_2$ of $v^*$.
		Since $p_{v^*}(u_1) = p_{v^*}(u_2)$, it holds that $u_1, u_2 \in V_2$ and that they share their second neighbour, i.e.\  $N_F(u_1) = N_F(u_2)$.
		Hence, \cref{lem:remove-twins} applies and $F$ is not minimal.

		\item The image of $p_{v^*}$ contains one vertex in $X$ and one vertex in $Y$.
		
		By injectivity of $p_{v^*}$ on $N_F(v^*) \setminus V_2$, all but one neighbour of $v^*$ are in $V_2$.
		Since every fibre in $X$ contains an even number of neighbours of $v^*$, there must be a fibre containing two neighbours $u_1, u_2$ of $v^*$ from $V_2$.
		As in the previous case, $N_F(u_1) = N_F(u_2)$ and $F$ is not minimal by \cref{lem:remove-twins}.

		\item The image of $p_{v^*}$ contains two vertices $v_1, v_2$ in $Y$.
		
		In this case, $N_F(v^*) \subseteq V_2$. Hence, all neighbours of $v^*$ are $\phi$-odd and in particular $v_1, v_2 \in Y_2$ lie in the same fibre.
		The set $C \coloneqq N_F(v^*) \cup \{v^*\}$ is a connected component of $F - \{v_1, v_2\}$.
        Since $v^*$ is not a cut vertex of $F$, the graph
        $F - \{v_1, v_2\}$ contains at least two connected components.
        Since the graph $K_{2,h}$ is $2$-connected,
        \cref{lem:cutvertex} implies that $F$ is not minimal.
	\end{enumerate}

	\item $v^*$ is $\phi$-odd.
	
	The vertex $v^*$ has at least one neighbour $u_i \in X_i$ in every fibre $i \in [h]$.
	In particular, it is of degree at least $h$.
	Distinguish further cases:
	\begin{enumerate}
		\item The image of $p_{v^*}$ has size one.
		
		In this case, the vertices $v^*, u_1, \dots, u_h$ together with the vertex in the image of $p_{v^*}$ induce a subgraph $K_{2,h}$ of $F$.
		
		\item The image of $p_{v^*}$ contains one vertex in $X$ and one vertex $v$ in $Y$.
		
By injectivity of $p_{v^*}$ on $N_F(v^*) \setminus V_2$, $v^*$ has at least $h-1$ neighbours $u_2, \dots, u_h$ in~$V_2$.
		\begin{enumerate}
			\item If $h$ is even, then $F$ is Eulerian by \cref{cl:eulerian}.
			The same remains true if $u_2, \dots, u_h$ (an odd number of vertices) are deleted and an edge between $v$ and $v^*$ is inserted. 
			Hence, there exists a path in $F$ from $v$ to $v^*$ avoiding the vertices $u_2, \dots, u_h$.
			This path can be contracted to a path of length two to yield a $K_{2, h}$-minor of $F$.
			\item If $h$ is odd, consider the subgraph $H \coloneqq F[Y_1 \cup Y_2 \cup X_1 \cup \dots \cup X_{h-1}]$.
			By \cref{cl:eulerian}, $H$ is Eulerian.
			Deleting $u_2, \dots, u_{h-1}$ (an odd number of vertices) from $H$ and adding an edge between $v$ and~$v^*$ yields another Eulerian graph.
			Thus, there exists a path in $H$ from $v$ to $v^*$ avoiding $u_2, \dots, u_{h-1}$.
			As above, this yields a minor $K_{2, h-1}$ in~$H$ and, 
			together with the neighbour $u_h \in X_h$ of $v^*$, a minor $K_{2, h}$ in $F$.
		\end{enumerate}

		\item The image of $p_{v^*}$ contains two vertices $v_1, v_2$ in $Y$.

		By \cref{cl:cutvertex}, $v^*$ is not a cut vertex of $F$.
		Hence, there exists a path in $F$ from $v_1$ to $v_2$ avoiding $v^*$.
		This path can be contracted to yield a minor $K_{2,h}$ of $F$.
		\qedhere
	\end{enumerate}
	\end{enumerate}
\end{proof}

	\section{Excluding  Minors}
	Next, we use the tools developed in previous sections to prove \cref{thm:main-vortex-free-hadwiger,thm:genus-strictly-refining,thm:had-strictly-refining}.
Towards this end, we first consider graph classes of bounded genus, and subsequently graph classes of bounded vortex-free Hadwiger number.

\subsection{Bounded Genus}
\label{sec:genus-subsection} 

We recall the following definitions from \cite{robertson_excluding_2024}.
A \emph{surface} $\Sigma$ is a non-null, possibly disconnected, compact $2$-manifold with possibly null boundary.
For a surface $\Sigma$, write $\widehat\Sigma$ for the surface obtained from $\Sigma$ by pasting a closed disc onto each component of the boundary.
If $\Sigma$ is a connected surface and $\Sigma$ is orientable,
then its \emph{genus} is defined as the number of handles one must add to a $2$-sphere to obtain $\widehat\Sigma$.
If $\Sigma$ is non-orientable, then its \emph{genus} is the number of crosscaps one must add to  a $2$-sphere to obtain $\widehat\Sigma$.
The \emph{genus} of a disconnected surface is the sum of the genera of its components.
Finally, the \emph{genus} of a graph $G$ is the minimal $g$ such that $G$ can be drawn on a surface $\Sigma$ of genus $g$ without edge crossings.
Write $\mathcal{E}_g$ for the class of graphs of genus $\leq g$.

Since the forbidden minors characterising $\mathcal{E}_g$ for $g \geq 1$ are unknown \cite{myrvold_large_2018}, 
we do not work with $\mathcal{E}_g$ directly 
but rather with superclasses that are described by a concise list of forbidden minors.
Recall that we write $\mathcal{P}$ for the class of planar graphs.
Also recall that, by Wagner's theorem \cite{wagner_uber_1937,kuratowski_sur_1930}, a graph is planar if, and only if, it does not contain $K_5$ and $K_{3,3}$ as a minor.
For $k \geq 1$, a \emph{$k$\nobreakdash-Kuratowski graph} \cite{robertson_excluding_2024} is a graph with exactly $k$ connected components, each of which is isomorphic to $K_5$ or $K_{3,3}$.
We write $\mathcal{P}^{(k)}$ for the class of graphs that exclude all $k$-Kuratowski graphs as a minor.
It is known that graphs of genus $g$ exclude all $(g+1)$\nobreakdash-Kuratowski graphs as a minor, see also \cite[(1.3)]{robertson_excluding_2024}.

\begin{lemma}[\protect{\cite[Theorem~1]{battle_additivity_1962}}]
    \label{lem:genus-k-kuratowski}
    For every $g \geq 0$,  it holds that $\mathcal{E}_g \subseteq \mathcal{P}^{(g+1)}$.
\end{lemma}

We show that $\mathcal{P}^{(g)}$ is closed under weak oddomorphisms and homomorphism distinguishing closed by applying the following lemma.

    \begin{lemma}\label{thm:minor-disjoint-union}
        Let $M_1,\dots,M_\ell$ be connected graphs, and let $\mathcal F$ be the class of $\{M_1,\dots,M_\ell\}$-minor-free graphs.
        For $k \geq 1$, 
        let $\mathcal F^{(k)}$ denote the class of graphs that exclude all minors that are the disjoint union of exactly $k$ not necessarily distinct graphs in $\{M_1,\dots,M_\ell\}$.
        If $\mathcal F$ is closed under weak oddomorphisms, then so is $\mathcal F^{(k)}$.
    \end{lemma}
    \begin{proof}
        Let $F \in \mathcal{F}^{(k)}$ and $\phi \colon F \to G$ be an oddomorphism, see \cref{lem:oddo-woddo}.
        If $G \not\in \mathcal{F}^{(k)}$,
        then there exists a minor $M_{i_1} + \dots + M_{i_k}$ of $G$.
        By \cref{lem:oddo-minors},
        there exists a minor $F'$ of $F$ admitting an oddomorphism $F' \to M_{i_1} + \dots + M_{i_k}$.
        By \cref{lem:oddo-connected},
        there exist pairwise disjoint subgraphs $F'_1, \dots, F'_k$ of $F'$ admitting oddomorphisms $F'_j \to M_{i_j}$ for $j \in [k]$.
        Since $\mathcal{F}$ is closed under oddomorphisms,
        it follows that $F'_j \not\in \mathcal{F}$ for all $j \in [k]$.
        Hence, each $F'_j$ contains at least one of the forbidden minors defining $\mathcal{F}$.
        In particular, $F' \not\in \mathcal{F}^{(k)}$,
        which implies that $F \not\in \mathcal{F}^{(k)}$, a contradiction.
    \end{proof}

    By \cref{lem:oddo-planar-degree},
    the class of planar graphs is closed under oddomorphisms.
    Hence, the following \cref{cor:p-k-closed} follows from \cref{thm:minor-disjoint-union,lem:oddo-connected,thm:rob62}.
    
    \begin{corollary}
        \label{cor:p-k-closed}
        For every $k \geq 1$, the class $\mathcal{P}^{(k)}$ is closed under weak oddomorphisms.
        In particular, the disjoint union closure of 
        $\mathcal{P}^{(k)}$ is homomorphism distinguishing closed.
    \end{corollary}

    \cref{lem:genus-k-kuratowski,cor:p-k-closed} imply that, for every $g \geq 0$,
    there exist non-isomorphic graphs $G$ and $H$ that are homomorphism indistinguishable over $\mathcal{E}_g$.
    Refining this, \cref{thm:genus-strictly-refining} shows that the relations $\equiv_{\mathcal{E}_g}$ for $g \geq 0$ form a strictly refining chain of graph isomorphism relaxations.

    \thmGenusStrictlyRefining*

    \begin{proof}

        Let $F$ be the graph obtained by taking $g+1$ copies of $K_5$ and connecting one of the vertices in each copy to a fresh vertex~$x$.
        Since $F$ contains a $(g+1)$\nobreakdash-Kuratowski graph as a minor, 
        $F \not\in \mathcal{P}^{(g+1)}$.
        In particular, by \cref{cor:p-k-closed,lem:genus-k-kuratowski}, we conclude that no graph in $\mathcal{E}_g$ admits a weak oddomorphism to $F$. 
        
        By \cite[Theorem~1]{battle_additivity_1962}, the genus of $F$ is the sum of the genera of its blocks.
        Since the blocks of $F$ are the $g+1$ copies of $K_5$ and the $g+1$ edges incident to $x$,
        it follows that $F$ has genus $g+1$.
        Now, the theorem follows  from \cref{lem:strictness-via-oddo}.
\end{proof}

\subsection{Bounded Vortex-Free Hadwiger Number}
\label{sec:extensions}

Next, we generalise \cref{thm:genus-strictly-refining} 
to graphs of bounded vortex-free Hadwiger number.
For $d,g \geq 0$, we write $\mathcal{C}_{g,d}$ for the class of all graphs $G$ such that $\dd_{\mathcal{E}_g}(G) \leq d$, i.e., there is a set $S \subseteq V(G)$ of size at most $d$ such that $G - S \in \mathcal{E}_g$.
Also, we write $\mathcal{C}_{g,d}^{\oplus s}$ for the closure of $\mathcal{C}_{g,d}$ under clique-sums of size at most $s$.
The \emph{vortex-free Hadwiger number} \cite{thilikos_killing_2022} of a graph $G$ is the minimum $k \geq 0$ such that $G \in \mathcal{C}_{k,k}^{\oplus k}$.
Recall that we write $\mathcal{H}_k \coloneqq \mathcal{C}_{k,k}^{\oplus k}$ for the class of graphs of vortex-free Hadwiger number at most $k$.

\cref{lem:genus-k-kuratowski} generalises to also allow for apices.
Alternatively, it is also possible to use the tools from \cref{sec:deletion-distance} to deal with apices.

\begin{lemma}
    \label{lem:genus-apices-k-kuratowski}
    For every $g, d \geq 0$,
    it holds that $\mathcal{C}_{g,d} \subseteq \mathcal{P}^{(g+1+d)}$.
\end{lemma}

\begin{proof}
    This directly follows from \cref{lem:genus-k-kuratowski}.
    Indeed, if a graph $G \in \mathcal{C}_{g,d}$ contains a $(g+1+d)$-Kuratowski graph as a minor, and $S \subseteq V(G)$ is a set of size at most $d$, then $G - S$ contains a $(g+1)$\nobreakdash-Kuratowski graph as a minor.
\end{proof}

Now, we can put everything together to obtain the following.

\thmHadStrictlyRefining*

\begin{proof}
    Let $G' \coloneqq (2k+1)K_5$ denote the disjoint union of $2k+1$ many $K_5$, and let $G$ be the graph obtained from $G'$ by adding universal vertices $u_1,\dots,u_{k+1}$, i.e., each $u_i$ is connected to all other vertices of $G$.
    Observe that $G$ is $(k+1)$-connected.
    Indeed, let $S \subseteq V(G)$ be a set of size at most $k$.
    Then there is some $i \in [k+1]$ such that $u_i \notin S$, and we obtain that $G - S$ is connected since $u_i$ is connected to all other vertices of $G$.
    
    We first claim that there is no graph $F \in \mathcal{H}_k$ that admits an oddomorphism to $G$.
    Suppose towards a contradiction that such a graph $F \in \mathcal{H}_k$ exists.
    Since $G$ is $(k+1)$-connected, by \cref{lem:oddo-to-kn-clique-sum}, we may assume that $F \in \mathcal{C}_{k,k}$.
    So $F \in \mathcal{P}^{(2k+1)}$ by \cref{lem:genus-apices-k-kuratowski}.
    By \cref{cor:p-k-closed}, we conclude that $G \in \mathcal{P}^{(2k+1)}$.
    But this is a contradiction, since $(2k+1)K_5$ is a subgraph of~$G$.

    Next, we argue that $G \in \mathcal{H}_{k+1}$.
    Indeed, $G$ is a $(k+1)$-sum of $2k+1$ many copies of $K_{k+6}$.
    Since $\mathcal{H}_{k+1}$ is closed under $(k+1)$-sums, it suffices to show that $K_{k+6} \in \mathcal{C}_{k+1,k+1}$.
    But this follows directly from the fact that $K_5 \in \mathcal{E}_1 \subseteq \mathcal{E}_{k+1}$, i.e.\ the graph $K_5$ can be drawn on the torus.
    Overall, the statement now follows from \cref{lem:strictness-via-oddo}.
\end{proof}

\cref{thm:main-vortex-free-hadwiger} is a direct consequence of \cref{thm:had-strictly-refining} since $G$, $H$ are in particular non-isomorphic. 	
	\section{Beyond Excluding Minors}
	In this section,
we give further evidence of the special role of minor-closed graph classes in the context of homomorphism indistinguishability, see \cite{seppelt_logical_2024}.
In particular, we show that \cref{conj:weak-roberson} fails when generalised to topological minors.
The key technical insight is the following \cref{lem:contractors-for-minors}, which also shows that the minor-closedness assumption in \cref{cor:two-cliquesum} is necessary. 

\begin{lemma}\label{lem:contractors-for-minors}
    If a graph class $\mathcal{F}$ is 
    closed under taking subgraphs and
    $2$-sums with $K_3$,
    then its homomorphism distinguishing closure $\cl(\mathcal{F})$
    is minor-closed.
\end{lemma}

Moreover, if \cref{conj:strong-roberson} holds and $\mathcal{F}$ is union-closed,
then $\cl(\mathcal{F})$ is in fact equal to the minor closure of $\mathcal{F}$. 
Before proving \cref{lem:contractors-for-minors},
we first discuss some of its consequences including a proof of \cref{thm:k5-top-minor}.

Let $\mathcal{D}_3$ denote the class of graphs of maximum degree $\leq 3$.
We write $\mathcal{D}_3^*$ for the closure of $\mathcal{D}_3$ under repeated $2$-sums with triangles.
More formally, a graph $G$ is contained in~$\mathcal{D}_3^*$ if, and only if, 
there is a tree decomposition $(T,\beta)$ of $G$ with a designated root node $r \in V(T)$ such that
\begin{enumerate}
    \item the torso $G\llbracket \beta(r) \rrbracket$ has maximum degree $\leq 3$, and
    \item $|\beta(t)| \leq 3$ for every $r \neq t \in V(T)$.
\end{enumerate}

\begin{lemma}
    \label{lem:d3-star}
    Two graphs are isomorphic if, and only if, they are homomorphism indistinguishable over~$\mathcal{D}_3^*$.
\end{lemma}

\begin{proof}[Proof (assuming \cref{lem:contractors-for-minors})]
    The class $\mathcal{D}_3^*$ is closed under taking subgraphs and 2\nobreakdash-sums with $K_3$.
    Hence, $\cl(\mathcal{D}_3^*)$ is minor-closed by \cref{lem:contractors-for-minors}.
    Since $\mathcal{D}_3 \subseteq \mathcal{D}_3^* \subseteq \cl(\mathcal{D}_3^*)$ and every graph is a minor of a graph of maximum degree $3$, we conclude that $\cl(\mathcal{D}_3^*)$ contains all graphs.
    By \cite{lovasz_operations_1967}, $\equiv_{\mathcal{D}_3^*}$ is the isomorphism relation $\cong$.
\end{proof}

Now, we directly obtain \cref{thm:k5-top-minor} by observing that every graph in $\mathcal{D}_3^*$ excludes $K_5$ as a topological minor.

\thmKfive*

Note that $K_5$ is optimal here since $K_4$-topological-minor-free graphs are precisely those of treewidth $\leq 2$, see \cref{thm:tw2}.
Moreover, it can also be easily verified that $\mathcal{D}_3^*$ has bounded $\infty$\nobreakdash-admis\-sibility (see \cite{dvorak_constant-factor_2013,siebertz_generalized_2025}) and bounded local treewidth (see \cite{eppstein_subgraph_1995,grohe_local_2003}), which implies that \cref{conj:weak-roberson} also cannot be generalised to those sparsity notions, see \cref{fig:graph-classes}.

The rest of this section is dedicated to the proof of \cref{lem:contractors-for-minors}.
Essentially, the proof is based on two ingredients:
First, we use properties of the homomorphism distinguishing closure established in \cite{seppelt_logical_2024} to reduce \cref{lem:contractors-for-minors} to the statement that, if $F \in \mathcal{F}$, then any graph $F'$ obtained from~$F$ by contracting edges is in the homomorphism distinguishing closure $\cl(\mathcal{F})$ of $\mathcal{F}$.
Second, 
we employ the so-called \emph{contractors} of
Lov\'asz and Szegedy~\cite{lovasz_contractors_2009} to simulate homomorphism counts from $F'$ by homomorphism counts from graphs obtained from $F$ by replacing the to-be-contracted edges by suitable series-parallel gadgets.
When attaching these gadgets, we rely on the assumption that $\mathcal{F}$ is closed under $2$-sums with $K_3$.

In order to state these arguments formally,
we recall the framework of bilabelled graphs and homomorphism matrices from 
\cite{mancinska_quantum_2020,grohe_homomorphism_2025}.
We follow the notational conventions of \cite[Section~3.2]{seppelt_homomorphism_2024}.

A \emph{bilabelled graph} is a tuple $\boldsymbol{F} = (F, u, v)$ where $F$ is a graph and $u,v\in V(F)$.
Given a graph $G$, the \emph{homomorphism matrix} of $\boldsymbol{F}$ with respect to $G$ is $\boldsymbol{F}_G \in \mathbb{N}^{V(G) \times V(G)}$ where $\boldsymbol{F}_G(x,y)$ for $x,y \in V(G)$ 
is the number of homomorphisms $h \colon F \to G$ such that $h(u) = x$ and $h(v) = y$.
For example, the identity matrix is $\boldsymbol{I}_G$ where $\boldsymbol{I} = ((\{u\}, \emptyset), u, u)$
and the adjacency matrix is $\boldsymbol{A}_G$ where $\boldsymbol{A} = ((\{u,v\}, \{uv\}), u, v)$. 
The all-ones matrix is $\boldsymbol{J}_G$ where $\boldsymbol{J} = ((\{u,v\}, \emptyset), u, v)$, see \cref{fig:bilabelled}.

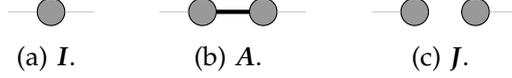
\begin{figure}
        \tikzset{
        	vertex/.style={draw,circle,fill=gray!80},
        	every node/.style={anchor=center},
        	lbl/.style={color=gray!40}
        }
        \centering
		\subcaptionbox{$\boldsymbol{I}$.}{
			\centering
			\begin{tikzpicture}[node distance=.7cm]
				\node (a) [vertex] {};
				\node (a1) [lbl,left of=a] {};
                \node (a2) [lbl,right of=a] {};
				\draw [lbl] (a2) -- (a) -- (a1);
			\end{tikzpicture}
		}
        \subcaptionbox{$\boldsymbol{A}$.}{
			\centering
			\begin{tikzpicture}[node distance=.7cm]
				\node (a) [vertex] {};
                \node (b) [vertex] at (.8,0) {};
				\node (a1) [lbl,left of=a] {};
                \node (a2) [lbl,right of=b] {};
				\draw [lbl] (a) -- (a1);
                \draw [lbl] (b) -- (a2);
                \draw [ultra thick] (a) -- (b);
 			\end{tikzpicture}
		}
        \subcaptionbox{$\boldsymbol{J}$.}{
			\centering
			\begin{tikzpicture}[node distance=.7cm]
				\node (a) [vertex] {};
                \node (b) [vertex] at (.8,0) {};
				\node (a1) [lbl,left of=a] {};
                \node (a2) [lbl,right of=b] {};
				\draw [lbl] (a) -- (a1);
                \draw [lbl] (b) -- (a2);
 			\end{tikzpicture}
		}
		
		\caption{Bilabelled graphs used in the proof of \cref{lem:contractors-for-minors} in wire notation of \cite{mancinska_quantum_2020}. The gray wires to the left/right indicate the position of the first/second label. Actual edges are depicted in black.}
		\label{fig:bilabelled}
	\end{figure}

Bilabelled graphs admit a wealth of combinatorial operations, which correspond to certain algebraic operations on their homomorphism matrices, see \cite[Section~3.2]{seppelt_homomorphism_2024}:

\emph{Unlabelling} a bilabelled graph $\boldsymbol{F} = (F, u, v)$ corresponds to the \emph{sum-of-entries} of the homomorphism matrix $\boldsymbol{F}_G$ in the sense that $\soe(\boldsymbol{F}_G) \coloneqq \sum_{x,y \in V(G)} \boldsymbol{F}_G(x,y) = \hom(F, G)$.
For this reason, we write $\soe(\boldsymbol{F}) \coloneqq F$ for the underlying unlabelled graph of $\boldsymbol{F}$.

For two bilabelled graphs $\boldsymbol{F} = (F, u, v)$ and $\boldsymbol{F'} = (F', u', v')$, the \emph{parallel composition} of $\boldsymbol{F}$ and $\boldsymbol{F'}$ is the bilabelled graph obtained from taking the disjoint union of $F$ and $F'$, identifying $u,u'$ and $v,v'$, and placing the labels on the two resulting vertices.
We write $\boldsymbol{F} \odot \boldsymbol{F}'$ to denote the parallel composition of $\boldsymbol{F}$ and $\boldsymbol{F'}$.
This corresponds to taking \emph{entrywise product} of the homomorphism matrices, i.e., $(\boldsymbol{F} \odot \boldsymbol{F'})_G(x,y) = \boldsymbol{F}_G(x,y) \boldsymbol{F'}_G(x,y) \eqqcolon (\boldsymbol{F}_G \odot \boldsymbol{F'}_G)(x,y)$ for every graph $G$ and every $x,y \in V(G)$. 

The \emph{series composition} of $\boldsymbol{F}$ and $\boldsymbol{F'}$ is the bilabelled graph $\boldsymbol{K} = (K,u,v')$ where $K$ is obtained from taking the disjoint union of $F$ and $F'$, and identifying the vertices $v,u'$.
We write $\boldsymbol{F} \cdot \boldsymbol{F}'$ to denote the series composition of $\boldsymbol{F}$ and $\boldsymbol{F'}$.
Series composition corresponds to taking the \emph{matrix product} of the homomorphism matrices, i.e., $(\boldsymbol{F} \cdot \boldsymbol{F'})_G(x,y) = \sum_{z \in V(G)} \boldsymbol{F}_G(x,z)\boldsymbol{F'}_G(z,y)  \eqqcolon (\boldsymbol{F}_G \cdot \boldsymbol{F'}_G)(x,y)$ for every graph $G$ and every $x,y \in V(G)$.

A bilabelled graph $\boldsymbol{S} = (S, u, v)$, where $u \neq v$, is \emph{series-parallel} if it can be constructed by a sequence of parallel and series compositions from the bilabelled graph $\boldsymbol{A}$.
We write $\mathcal{S}$ for the class of bilabelled series-parallel graphs.
It is well-known that, if $\boldsymbol{S} = (S, u, v)$ is series-parallel, then there is a tree decomposition $(T,\beta)$ of $S$ of width $\leq 2$ such that $u,v \in \beta(t)$ for some node $t \in V(T)$.
In particular, we obtain the following:

\begin{observation}\label{obs:series-parallel}
    Let $\mathcal{F}$ be a graph class that is closed under taking subgraphs and $2$-sums with $K_3$.
    Let $F \in \mathcal{F}$ and $uv \in E(F)$, and set $\boldsymbol{F} = (F,u,v)$.
    For every $\boldsymbol{S} \in \mathcal{S}$,
    it holds that $\soe(\boldsymbol{F} \odot \boldsymbol{S}) \in \mathcal{F}$.
\end{observation}
    
A key tool in the proof of \cref{lem:contractors-for-minors} is the notion of a \emph{contractor}.
Let $G$ be a graph.
A \emph{contractor} \cite{lovasz_contractors_2009} for $G$ is a finite linear combination $\sum_{\boldsymbol{F}} \alpha_{\boldsymbol{F}} \boldsymbol{F}$ of bilabelled graphs $\boldsymbol{F} = (F, u, v)$ whose labels do not coincide, i.e.\ $u \neq v$, such that $\boldsymbol{I}_G = \sum_{\boldsymbol{F}} \alpha_{\boldsymbol{F}} \boldsymbol{F}_G$.
Since parallel composition with $\boldsymbol{I}$ amounts to identifying the two labelled vertices, 
a contractor simulates label identification by homomorphism counts from patterns in which the labels are not contracted.
We rely on the following result establishing the existence of series-parallel contractors.

\begin{lemma}[Lov\'asz and Szegedy \protect{\cite[Theorem 1.4]{lovasz_contractors_2009}, see also \cite[Supplement~6.29b]{lovasz_large_2012}}]\label{lem:lovasz-szegedy}
    For all graphs $G$ and $H$,
    there exists a finitely-supported coefficient vector $\alpha_{\boldsymbol{S}} \in \mathbb{Q}$ such that
    \[
        \boldsymbol{I}_G = \sum_{\boldsymbol{S} \in \mathcal{S}} \alpha_{\boldsymbol{S}} \boldsymbol{S}_G
        \quad \text{ and } \quad
        \boldsymbol{I}_H = \sum_{\boldsymbol{S} \in \mathcal{S}} \alpha_{\boldsymbol{S}} \boldsymbol{S}_H.
    \]
\end{lemma}

Equipped with \cref{lem:lovasz-szegedy}, we now conduct the proof of \cref{lem:contractors-for-minors}.

\begin{proof}[Proof of \cref{lem:contractors-for-minors}]
    We start by showing the following claim.
    
    \begin{claim}\label{claim:contractor}
        Let $\mathcal{F}'$ be a graph class that is closed under taking subgraphs and $2$-sums with~$K_3$.
        Let $F \in \mathcal{F}'$, and let $F'$ be the graph obtained from $F$ by contracting an edge $uv \in E(F)$.
        Then $F' \in \cl(\mathcal{F}')$.
    \end{claim}
    
    \begin{claimproof}
        Contrapositively,
        let $G$ and $H$ be two graphs with $\hom(F',G) \neq \hom(F',H)$.
        Let $\alpha_{\boldsymbol{S}} \in \mathbb{Q}$ for $\boldsymbol{S} \in \mathcal{S}$ denote the finitely-supported coefficient vector from \cref{lem:lovasz-szegedy}.
        Write $F^-$  for the graph obtained from $F$ by deleting the edge $uv$ and let $\boldsymbol{F}^- \coloneqq (F^-, u, v)$ as well as $\boldsymbol{F} \coloneqq (F, u, v)$.
        The parallel composition $\boldsymbol{F}^- \odot \boldsymbol{I}$ of $\boldsymbol{F}^-$ with $\boldsymbol{I}$ amounts to identifying the two vertices $u,v$, i.e., contracting  the edge $uv$ in $F$.
        Hence, $\soe(\boldsymbol{F}^- \odot \boldsymbol{I}) \cong F'$.
        By \cref{lem:lovasz-szegedy},
        \[\hom(F', G) = \soe\left(( \boldsymbol{F}^- \odot \boldsymbol{I})_G\right) = \sum \alpha_{\boldsymbol{S}} \soe\left((\boldsymbol{F}^- \odot \boldsymbol{S})_G\right)\]
        and analogously
        \[\hom(F', H) = \soe\left(( \boldsymbol{F}^- \odot \boldsymbol{I})_H\right) = \sum \alpha_{\boldsymbol{S}} \soe\left((\boldsymbol{F}^- \odot \boldsymbol{S})_H\right).\]
        Since $\hom(F',G) \neq \hom(F',H)$, there exists some $\boldsymbol{S} \in \mathcal{S}$ such that
        $\soe\left((\boldsymbol{F}^- \odot \boldsymbol{S})_G\right) \neq \soe\left((\boldsymbol{F}^- \odot \boldsymbol{S})_H\right)$.
        Note that $\soe(\boldsymbol{F}^- \odot \boldsymbol{S})$ is a subgraph of $\soe(\boldsymbol{F} \odot \boldsymbol{S})$, which lies in $\mathcal{F}'$ by \cref{obs:series-parallel}.
        Since $\mathcal{F}'$ is closed under taking subgraphs, also $\soe(\boldsymbol{F}^- \odot \boldsymbol{S}) \in \mathcal{F}'$.
        Hence, $G \not\equiv_{\mathcal{F}'} H$.
        We conclude that $F' \in \cl(\mathcal{F}')$.
    \end{claimproof}

    For $\ell \geq 0$, 
    let $\mathcal{F}_\ell$ denote the class of graphs obtained from $\mathcal{F}$ by performing at most $\ell$ edge contractions.
    We show by induction that $\mathcal{F}_\ell \subseteq \cl(\mathcal{F})$ for all $\ell \geq 0$.
    The base case $\ell = 0$ is trivial since $\mathcal{F}_0 = \mathcal{F}$.

    So suppose $\ell \geq 0$ and let $F' \in \mathcal{F}_{\ell+1}$.
    By definition, there is some $F \in \mathcal{F}_\ell$ such that $F'$ is the graph obtained from $F$ by contracting an edge $uv \in E(F)$.
    Note that $\mathcal{F}_\ell$ is closed under taking subgraphs and $2$-sums with $K_3$.
    Hence, by \cref{claim:contractor}, we conclude that $F' \in \cl(\mathcal{F}_\ell)$.
    It follows that
    \[\mathcal{F}_{\ell+1} \subseteq \cl(\mathcal{F}_\ell) \subseteq \cl(\cl(\mathcal{F})) = \cl(\mathcal{F}),\]
    where the second inclusion holds by the induction hypothesis and \cref{eq:subset-closure}, and the final equality holds by \cref{eq:double-closure}.

    Now, let $\mathcal{F}^*$ be the minor closure of $\mathcal{F}$, i.e., $\mathcal{F}^*$ consists of the graphs $F^*$ that are minors of some graph $F \in \mathcal{F}$.
    Since, for every graph $F^* \in \mathcal{F}^*$, there is some $\ell \geq 0$ such that $F^* \in \mathcal{F}_\ell$, we conclude that $\mathcal{F}^* \subseteq \cl(\mathcal{F})$.
    In particular, by \cref{eq:subset-closure,eq:double-closure},
    $\cl(\mathcal{F})  \subseteq \cl(\mathcal{F}^*) \subseteq \cl(\cl(\mathcal{F})) = \cl(\mathcal{F})$.
    Hence, $\cl(\mathcal{F}) = \cl(\mathcal{F}^*)$.
    Finally, by \cite[Theorem~8]{seppelt_logical_2024},
    the homomorphism distinguishing closure of a minor-closed graph class is itself minor-closed.
\end{proof}

	\section{Conclusion}
	We proved a series of results on the distinguishing power of homomorphism indistinguishability relations over sparse graph classes.
Most notably, we proved \cref{conj:weak-roberson} for all vortex-free graph classes.
In particular, homomorphism indistinguishability over graphs of bounded Euler genus is not the same as isomorphism.
As a negative result, we showed that \cref{conj:weak-roberson} cannot be significantly generalised beyond graph classes with an excluded minor.
Finally, 
we showed that several graph classes are homomorphism distinguishing closed 
and separated homomorphism indistinguishability relations defined via the genus or the vortex-free Hadwiger number.

Nevertheless, several questions remain open.
Most notably, \cref{conj:weak-roberson} remains unresolved for $k \geq 7$.
With \cref{thm:main-vortex-free-hadwiger} in mind, it remains to understand the role of vortices.
Is homomorphism indistinguishability over planar graphs with one vortex the same as isomorphism?
Towards this end, it may be helpful to obtain a purely graph-theoretic proof of \cref{lem:oddo-planar-degree}\ref{it:planar-oddo}, for which the only known proof crucially relies on the characterisation via quantum isomorphisms \cite{mancinska_quantum_2020,atserias_quantum_2019}.

Finally, our results separating homomorphism indistinguishability relations $\equiv_{\mathcal{F}}$ from isomorphism
pave the way for studying the computational complexity of determining whether $G \equiv_{\mathcal{F}} H$ for input graphs $G$, $H$, and fixed $\mathcal{F}$.
This problem has been conjectured to be undecidable for every proper minor-closed graph class $\mathcal{F}$ of unbounded treewidth \cite[Conjecture 25]{seppelt_algorithmic_2024}.
However, up to this point, this is only known for the class of planar graphs 
\cite{slofstra_set_2019,atserias_quantum_2019,mancinska_quantum_2020}.

	\section*{Acknowledgements}
	\textit{Tim Seppelt:} European Union (CountHom, 101077083). Views and opinions expressed are however those of the author(s) only and do not necessarily reflect those of the European Union or the European Research Council Executive Agency. Neither the European Union nor the granting
authority can be held responsible for them. 
	\section*{AI Statement}

	The proofs of \cref{thm:oddo-pw,thm:oddo-tw} were found by Claude Opus 5 after we supplied a similar proof of an approximate version of \cref{thm:oddo-tw}.

\bibliographystyle{plainurl}

\end{document}